\documentclass{article}
\usepackage{amssymb}
\usepackage{amsmath}

\usepackage{psfrag}
\usepackage[dvips]{graphicx}
\usepackage{epsfig}
\usepackage{wrapfig}

\newtheorem{theorem}{Theorem}

\newtheorem{lemma}{Lemma}

\newtheorem{proposition}[theorem]{Proposition}

\newenvironment{proof}[1][Proof]{\noindent\textbf{#1.} }{\ \rule{0.5em}{0.5em}}

\begin{document}
\title{\textbf{Boltzmann Equation\\ with a Large Potential in a Periodic Box}}
\author{\large{\textbf{Chanwoo Kim}}\\ \small{Department of Mathematics, Brown University, Providence, RI 02917, USA.
E-mail : ckim@math.brown.edu}}
\maketitle

\begin{abstract}
The stability of the Maxwellian of the Boltzmann equation with a large amplitude external potential $\Phi$ has been an important open problem. In this paper, we resolve this problem with a large $C^3-$potential in a periodic box $\mathbb{T}^d$, $d\geq 3$. We use \cite{Asano94} in $L^p-L^{\infty}$ framework to establish the well-posedness and the $L^{\infty}-$stability of the Maxwellian $\mu_E (x,v) = \exp \left\{-\frac{|v|^2}{2}-\Phi(x)\right\}$.
\end{abstract}

\maketitle
\section{Introduction}
In the presence of a potential $\Phi$, a density of a dilute charged gas is governed by the Boltzmann equation
\begin{equation}
\partial_t F + v \cdot\nabla_x F -\nabla_x \Phi(x)\cdot\nabla_v F = Q(F,F) \ \ , \ \ \ \ \ F(0,x,v) = F_0 (x,v),\label{Boltzmann}
\end{equation}
where $F(t,x,v)$ is a distribution function for the gas particles at a time $t\geq 0,$ a position $x\in \mathbb{T}^d \ $and a velocity $v\in \mathbb{R}^{d}$ for $d\geq 3$. Here, the external potential $\Phi(x)$ is a given function only depends on the spatial variable $x$ in a periodic box $\mathbb{T}^d$. The collision operator $Q$ takes the form%
\begin{eqnarray}
Q(F_{1},F_{2})
&=&\int_{\mathbb{R}^{d}}\int_{\mathbb{S}^{d-1}} B(v-u,\omega)
F_{1}(u^{\prime })F_{2}(v^{\prime })d\omega du
-\int_{\mathbb{R}^{d}}\int_{\mathbb{S}^{d-1}} B(v-u,\omega) F_{1}(u)F_{2}(v)d\omega du   \label{qgl}\\
&\equiv &Q_{+}(F_{1},F_{2})-Q_{-}(F_{1},F_{2}) \ ,
\notag
\end{eqnarray}%
where $u^{\prime }=u+[(v-u)\cdot \omega ]\omega ,$ $v^{\prime
}=v-[(v-u)\cdot \omega ]\omega $ and
$
B(v-u,\omega) = |v-u|^{\gamma} q_0 (\frac{v-u}{|v-u|}\cdot \omega),
$
with $0< \gamma \leq 1$ (hard potential) and
$
\int_{\mathbb{S}^{d-1}} q_0 (\hat{u}\cdot \omega) d\omega < +\infty \ \text{(angular cutoff)}\label{Gradcutoff}
$
for all $\hat{u}\in\mathbb{S}^{d-1}$.

Throughout the paper, we study the stability of a local Maxwellian for given potential $\Phi$ :
\begin{equation}
\mu_E(x,v) \ = \ \exp\left\{ -\frac{|v|^2}{2} - \Phi (x)
\right\} \ = \ \mu(v) e^{-\Phi (x)} \ , \label{lambda}
\end{equation}
where $\mu(v)= \exp\left\{ -\frac{|v|^2}{2} \right\}$ is the standard global Maxwellian in the no potential case, $\Phi \equiv 0$ (\cite{GuoSoft}).
Define a perturbation distribution $f=f(t,x,v)$ by
\begin{equation}
F(t,x,v) = \mu_E(x,v) + \sqrt{\mu_E(x,v)}f(t,x,v).\label{perturb}
\end{equation}
Then the equation for the perturbation $f$ is
\begin{equation}
\partial_t f + v\cdot \nabla_x f -\nabla\Phi\cdot\nabla_v f + e^{-\Phi(x)}Lf = e^{-\frac{1}{2}\Phi(x)}\Gamma(f,f) \ \ , \ \ \ \ \ f(0,x,v) = f_0 (x,v)= \frac{F_0-\mu_E}{\sqrt{\mu_E}} \ ,\label{linearBE}
\end{equation}
where the standard operators of the linearized Boltzmann theory (\cite{GuoVPB}) are
\begin{equation*}
Lf \equiv \nu f - Kf = -\frac{1}{\sqrt{\mu}}\{ Q(\mu,\sqrt{\mu}f) + Q(\sqrt{\mu}f,\mu)\} =\nu f - \int \mathbf{k}(v,v^{\prime}) f(v^{\prime}) dv^{\prime},
\end{equation*}
with the collision frequency $\nu(v)\equiv \int |v-u|^{\gamma} \mu(u) q_0 du d\omega \sim \{1+|v|\}^{\gamma}$ for $0< \gamma \leq 1$ ; and
\begin{equation*}
\Gamma(f_1,f_2) = \frac{1}{\sqrt{\mu}} Q(\sqrt{\mu}f_1, \sqrt{\mu}f_2) \equiv \Gamma_{+}(f_1,f_2)-\Gamma_{-}(f_1,f_2).
\end{equation*}

\subsection{External Potential and Conservation of Momentum}
Let $F$ be a solution of the Boltzmann equation (\ref{Boltzmann}) with an external potential $\Phi$. As the no potential case ($\Phi\equiv 0$), we have the (excess) conservations of mass and energy :
\begin{eqnarray}
\iint_{\mathbb{T}^d \times \mathbb{R}^d} \{F(t,x,v)-\mu_E\} dv dx &=& \iint_{\mathbb{T}^d \times \mathbb{R}^d} \{F_0(x,v)-\mu_E\} dv dx \ \equiv \ M_0,\label{masscons}\\
\iint_{\mathbb{T}^d \times \mathbb{R}^d} (\frac{|v|^2}{2} + \Phi(x) )\{F(t,x,v)-\mu_E\} dv dx &=& \iint_{\mathbb{T}^d \times \mathbb{R}^d} (\frac{|v|^2}{2} + \Phi(x) )\{F_0(x,v)-\mu_E\} dv dx \ \equiv \ E_0,\label{energycons}
\end{eqnarray}
as well as the excess entropy inequality :
\begin{equation}
\mathcal{H}(F(t))-\mathcal{H}(\mu_E) \leq \mathcal{H}(F_0) -\mathcal{H}(\mu_E),\label{entropy}
\end{equation}
where $\mathcal{H}(g)\equiv \iint g \ \ln g \ dv dx$.

However, in the presence of an external potential $\Phi$, the momentum conservation law is delicate. In general, the momentum is not conserved : multiplying the Boltzmann equation (\ref{Boltzmann}) by $v_i$ and integrating over $\mathbb{T}^d \times\mathbb{R}^d$, we have
\begin{eqnarray*}
\frac{d}{dt}\iint v_i F(t) dx dv - \iint v_i \nabla_x \Phi(x) \cdot \nabla_v F (t) dx dv =0.
\end{eqnarray*}
Applying the integration by part to the second term, we have
\begin{eqnarray*}
\frac{d}{dt}\iint v_i F(t) dx dv + \iint \partial_i \Phi(x) F (t) dx dv =0.
\end{eqnarray*}
If the potential $\Phi$ does not depend on $x_i$, then the second integration of the above equation vanishes. Therefore we have the conservation of momentum for $v_i$. Otherwise, in general, we do not have such a conservation law of momentum.

More precisely, define a map
\begin{equation*}
\Lambda : \mathbb{T}^d \ \rightarrow \ \{ \text{Linear Subspaces of }\mathbb{R}^d \} \ \ , \ \ \ \Lambda(x)=\text{span}\{\nabla\Phi(x)\} \equiv \{v\in\mathbb{R}^d : v=\tau \nabla\Phi(x) \ \ , \tau\in\mathbb{R}\}.
\end{equation*}
Define
\begin{equation}
\Lambda(\mathbb{T}^d) \ \equiv \ \bigcup_{x\in\mathbb{T}^d} \Lambda(x) \ ,
\end{equation}
which is a linear subspace of $\mathbb{R}^d$. Further we can decompose $\mathbb{R}^d = \Lambda(\mathbb{T}^d)\bigoplus \Lambda(\mathbb{T}^d)^{\bot}$. More precisely we define a \textbf{degenerate subspace} of $\nabla\Phi$ by
\begin{equation}
\Lambda(\mathbb{T}^d)^{\bot} \ \equiv \ \bigcap_{x\in\mathbb{T}^d} \Lambda(x)^{\bot},\label{v1tovn}
\end{equation}
which is a linear subspace of $\mathbb{R}^d$. Let $n= dim \ \Lambda(\mathbb{T}^d)^{\bot}$. Notice that generically the degenerate subspace of $\nabla\Phi$ is a zero space $\{0\}$ and $n= dim \ \Lambda(\mathbb{T}^d)^{\bot}=0$. Upon relabeling and reorienting the coordinates axes, we may assume that $\Lambda(\mathbb{T}^d)^{\bot}$ is spanned by $\{e_1,...,e_n\}$, i.e.
\begin{equation}
 \Lambda(\mathbb{T}^d)^{\bot}= \text{span}\{e_1,...,e_n\} \ \ , \ \ \ \ \ \ \ n= dim \ \Lambda(\mathbb{T}^d)^{\bot}.\label{n}
\end{equation}
If $\Phi$ is differentiable then $\partial_{x_1}\Phi = \cdots = \partial_{x_n}\Phi \equiv 0 $ and $\Phi= \Phi(x_{n+1},\cdots, x_d)$. Further we assume $\Phi \in C^3(\mathbb{T}^d)$ and satisfies the periodic boundary condition in $\mathbb{T}^d$ and $1\leq \Phi(x) < |\Phi|_{\infty}$. Then we have the (excess) conservation of momentum for degenerate $\{v_1,...,v_n\}$ :
\begin{eqnarray}
\iint_{\mathbb{T}^d \times \mathbb{R}^d} \{F(t,x,v)-\mu\} (v_1,\ldots v_n)^{T} \ dv dx  \ = \
\iint_{\mathbb{T}^d \times \mathbb{R}^d} \{F_0(x,v)-\mu\} (v_1,\ldots v_n)^{T} \ dv dx  \ = \
 \mathbf{J}_0\in \mathbb{R}^n.\label{momentumcons}
\end{eqnarray}
Notice that generically $ \ \Lambda(\mathbb{T}^d)^{\perp}$ ,the degenerate subspace of $\nabla\Phi$, is a zero space $\{0\}$ and $n= dim \ \Lambda(\mathbb{T}^d)^{\bot}=0$ so that we do not have such a momentum conservation law as (\ref{momentumcons}).

It is important to point out that the momentum conservation law (\ref{momentumcons}) is necessary in order to get decay (\ref{nondecay}) in Theorem 1. In particular the condition (\ref{momentumcons}) is used in (\ref{cruciall}) in order to show the crucial positivity of $L$ in (\ref{positivity}). Without the condition (\ref{momentumcons}), we have the stability result (\ref{stability}) in Theorem 2 but not a decay.

\subsection{Main Result}
We introduce the weight function for $\beta>d/2$,
\begin{equation}
w(x,v) = \left\{ \frac{|v|^2}{2} + \Phi(x)\right\}^{\beta/2}.\label{weight}
\end{equation}
\begin{theorem}
Assume that an external potential $\Phi$ is a periodic $C^3$-function on $\mathbb{T}^d$ and $\Phi=\Phi(x_{n+1},\cdots,x_d)$ for some $n\leq d$. Assume the conservations of mass (\ref{masscons}), energy (\ref{energycons}) and momentum for degenerate $\{v_1,\cdots,v_n\}$ (\ref{momentumcons}) are valid for $F_0 = \mu_E+\sqrt{\mu_E}f_0$ with
\begin{equation}
( \ M_0, \ E_0, \ \mathbf{J}_0 \ ) = ( \ 0 \ , \ 0 \ , \ \mathbf{0} \ ) \in \mathbb{R}\times\mathbb{R}\times\mathbb{R}^n. \label{conlaw}
\end{equation}
Then there exists $\delta>0$ such that if $F_0(x,v)=\mu_E + \sqrt{\mu_E}f_0(x,v)$ and $||wf_0||_{\infty}\leq \delta$, there exists a unique solution $F(t,x,v)=\mu_E+\sqrt{\mu_E}f(t,x,v)\geq 0$ for the Boltzmann equation (\ref{Boltzmann}) such that
\begin{equation}
\sup_{0\leq t\leq \infty} e^{\lambda t}||wf(t)||_{\infty} \leq C ||wf_0||_{\infty}.\label{nondecay}
\end{equation}
\end{theorem}

Without the conservation of momentum for degenerate $\{v_1,\cdots,v_n\}$ (\ref{momentumcons}), we are not able to prove the $L^{\infty}-$decay (\ref{nondecay}). The reason is that such a momentum conservation law is a crucial to show the positivity of $L$ in Proposition 4. However we have the $L^{\infty}-$stability using the natural excess entropy inequality (\ref{entropy}).

\begin{theorem}
Assume that an external potential $\Phi$ is a periodic $C^3$-function on $\mathbb{T}^d$. Assume the excess conservation of mass (\ref{masscons}), energy (\ref{energycons}) and the excess entropy inequality (\ref{entropy}) are valid for $F_0 = \mu_E+\sqrt{\mu_E}f_0$ with
\begin{equation}
|M_0| , \ |E_0| , \ |\mathcal{H}(F_0)-\mathcal{H}(\mu_E)| \ < \infty.
\end{equation}
Then there exists $\delta>0$ such that if $F_0(x,v)=\mu_E + \sqrt{\mu_E}f_0(x,v)$ and
\begin{equation}
||wf_0||_{\infty} + \sqrt{\mathcal{H}(F_0)-\mathcal{H}(\mu_E) + |M_0| + |E_0|} \ \leq \ \delta \ ,
\end{equation}
then there exists a unique solution $F(t,x,v)=\mu_E+\sqrt{\mu_E}f(t,x,v)\geq 0$ for the Boltzmann equation (\ref{Boltzmann}) such that
\begin{equation}
\sup_{0\leq t\leq \infty}||wf(t)||_{\infty} \ \leq \ C \left\{ \ ||wf_0||_{\infty} + \sqrt{\mathcal{H}(F_0)-\mathcal{H}(\mu_E) + |M_0| + |E_0|} \ \right\} \ .\label{stability}
\end{equation}
\end{theorem}
Notice that we do not need any smallness assumption for the external potential $\Phi$ in both theorems.
\\
\newline There are some investigations about the dynamical problems of the Boltzmann equation with an external potential. The local well-posedness was established in \cite{Drange} and \cite{Asano84}. Near Maxwellian regime, the global well-posedness was established in \cite{LY}, \cite{UYZ}, \cite{DUYZ}, \cite{Duan}, \cite{Yu} and \cite{D-S} with some smallness assumptions for the external potential $\Phi$ using the nonlinear energy method. In \cite{Tabata1}, using the semi-group approach, the global well-posedness was established with some smallness assumptions for the external potential $\Phi$ in a periodic box. This result was later generalized in \cite{Tabata2} and \cite{Tabata3} to the case of an unbounded external potential in $\mathbb{R}^3$ with spherically symmetric assumption. Near vacuum regime, the global well-posedness was established in \cite{GuoVacuum} with a small (self-consistent) external potential and in \cite{DYZ1} with a large external potential $\Phi$ with some special conditions. In the case of 1-dimensional Boltzmann equation ($x\in\mathbb{R}, \ v\in\mathbb{R}^3$) near Maxwellian regime, the well-posedness and stability are established in \cite{E-G-M} with a large amplitude external potential.

In the presence of a large amplitude external potential, the key difficulty is the collapse of Sobolev estimate in higher order energy norms. The derivatives of the Boltzmann solution can grow in time unless the potential is small. In order to overcome this difficulty, we use the weighted $L^{\infty}$ formulation without any derivatives(\cite{Guo08}\cite{Guo_short}).
\\
\newline With the conservation of momentum for degenerate $\{v_1,\cdots,v_n\}$, we use the $L^2 -L^{\infty}$ framework(\cite{Guo08}) which consists of two parts : First, establish $L^2$-decay for the linear Boltzmann equation(Section 3) ; Second, establish the $L^{\infty}$-decay for the nonlinear Boltzmann solution using the Vidav's idea and the $L^2$-decay of the first part(Section 4).

For proving the linear $L^2-$decay (Section 3), the main difficulty is the absence of the momentum conservation laws for all velocity components $\{v_1,\cdots,v_d\}$. The key ingredient to prove the linear $L^2-$decay is the positivity of the linear Boltzmann operator $L$ (Proposition 3). Following \cite{GuoVPB}, we establish such a positivity of $L$ by the contradiction argument. The consequential limiting function is non-zero and only has the hydrodynamic part which is the null space of $L$ spanned by the basis $\{ \ \sqrt{\mu}, \ v\sqrt{\mu}, \ |v|^2 \sqrt{\mu} \ \}$. The coefficients of the limiting function satisfy the macroscopic equation $(\ref{ME1})-(\ref{ME5})$. Using the conservation of momentum (\ref{momentumcons}) for degenerate $\{v_1,\cdots,v_n\}$ and the periodicity in $\mathbb{T}^d$ crucially, we are able to show that all the coefficients are zero, which is contradiction.

For nonlinear $L^{\infty}-$decay (Section 4), we use Vidav's idea. Denote $X(s;t,x,{v})$ the backward trajectory at time $s$, starting at time $t\in [0,\infty)$, position $x\in\mathbb{T}^d$ with velocity $v\in\mathbb{R}^d$. Similarly $X(s_1;s,X(s;t,x,{v}),{v}^{\prime})$ is the backward trajectory at time $s_1$ starting at time $s\in [0,t]$, position $X(s;t,x,{v})\in\mathbb{T}^d$ with velocity $v^{\prime}\in\mathbb{R}^d$. The goal is to establish the following estimate for the solution of the Boltzmann equation :
\begin{eqnarray}
\int_0^t ds \int_{\frac{1}{N} \leq|{v}^{\prime}|\leq N} d{v}^{\prime} \int_0^{s} ds_1 \int dv^{\prime\prime} \ |f(s_1,X(s_1;s,X(s;t,x,{v}),{v}^{\prime}),v^{\prime\prime})| \label{a1}\\
\lesssim \left(\varepsilon+\frac{1}{N}\right)\sup_{0\leq s_1\leq t}||f(s_1)||_{\infty} +\int_0^t ||f(s_1)||_{L^p} \ ds_1 \ , \ \ \ \ \ \ \ \ \ \label{a2}
\end{eqnarray}
where $p=1$ or $p=2$. Once we have the estimate (\ref{a2}) for $p=2$, using the established $L^2-$decay, we are able to show the $L^{\infty}-$decay. The basic idea to show the desired estimate (\ref{a2}) is to establish
\begin{equation}
\det \left\{ \frac{dX(s_1;s,X(s;t,x,{v} ),{v}^{\prime})}{d{v}^{\prime}}
\right\}\neq 0 \ , \label{nonvanishing}
\end{equation}
for almost every $(s_1,s, v^{\prime})\in (0,t)\times (0,t)\times \mathbb{R}^d$ for all $X(s;t,x,v)\in\mathbb{T}^d$. Then we apply the change of variables
\begin{equation*}
v^{\prime} \ \rightarrow \ X(s_1;s,X(s;t,x,{v} ),{v}^{\prime}) \ ,
\end{equation*}
for main part of $(0,t)\times(0,t)\times\mathbb{R}^d$ to bound (\ref{a1}) by the $L^2-$term in (\ref{a2}) and to bound (\ref{a1}) by $L^{\infty}-$term in (\ref{a2}) for the small remainder part. In this paper, we use Asano's result in \cite{Asano94} to verify the crucial condition (\ref{nonvanishing}) for smooth external potentials. In \cite{Asano94}, using the symplectic geometric approach, the points fail to satisfy the condition (\ref{nonvanishing}) is characterized by the eigenvalues of some symmetric matrix. Because of this formulation, using the standard min-max principle, Lemma \ref{asano} was established in \cite{Asano94}. The condition (\ref{nonvanishing}) has been proved in many other cases. In \cite{Guo08}, the condition (\ref{nonvanishing}) has been shown in the case of bounded domains $\Omega\subset \mathbb{R}^d$ with several boundary conditions without an external field($\nabla\Phi \equiv 0$). Notice that the characteristics are determined according to the boundary conditions. In the case of in-flow, bounce-back, diffuse reflection boundary conditions, the condition (\ref{nonvanishing}) was proved for bounded domains $\Omega\subset\mathbb{R}^3$ with the smooth boundary $\partial\Omega$. For the specular reflection boundary condition case, the condition (\ref{nonvanishing}) was established for analytic and strictly convex domains $\Omega\subset\mathbb{R}^3$. In \cite{kim1}, for the specular reflection case, the condition (\ref{nonvanishing}) was established for analytic and non-convex, 2-dimension domains $\Omega\subset\mathbb{R}^2$. Without boundary condition : $\Omega=\mathbb{R}^d$, in \cite{GH}, the condition (\ref{nonvanishing}) was shown for self-consistent electric fields if the perturbation $f$ is small. In 1-dimensional case $\Omega=\mathbb{R}$, the condition is proved for external potentials with large amplitude (\cite{E-G-M}).
\\
\newline  Without the conservation of momentum for degenerate $\{v_1,\cdots,v_n\}$, we are not able to prove the linear $L^{2}$-decay. Instead, we use the natural excess entropy inequality (\ref{entropy}) to obtain $L^1$-stability of the Boltzmann equation(\cite{Guo_short}). The basic idea is that
\begin{equation*}
\iint \{F \ln F - \mu_E \ln \mu_E \} \sim \iint (1+\ln \mu_E) (F-\mu_E) + \iint\frac{1}{2\mu_E}(F-\mu_E)^2 \ .
\end{equation*}
Notice that the first term is controlled via the excess entropy inequality (\ref{entropy}) and the second term is controlled via the conservation of mass (\ref{masscons}) and energy (\ref{energycons}). Therefore we can control the third term and the $L^1-$norm of $|F-\mu_E|$. Indeed, the $L^p-$term in (\ref{a2}) with $p=1$ is bounded by the mass and energy and entropy. Therefore, we obtain the $L^{\infty}-$boundedness(stability) of the Boltzmann solution $F$.
\\
\newline Our paper is organized as follows. In section 2, we state the Asano's result(Lemma 1) and construct an open covering of the points fail to satisfy (\ref{nonvanishing}). In section 3, we establish the linear $L^2$-decay (Theorem 3). In section 4, we use Vidav's idea to bootstrap the nonlinear $L^{\infty}$-decay (Theorem 1) from linear $L^2$-decay. In section 5, we use the entropy-energy estimate(\cite{Guo_short}) to prove the $L^{\infty}$-stability (Theorem 2).

\section{Characteristics and Transversality}
In this section we study the characteristcs for the Boltzmann equation with an external field (\ref{Boltzmann}). The hamiltonian of the system is given by
\begin{equation*}
H(x,v) = \frac{|v|^2}{2} + \Phi(x).
\end{equation*}
We consider the Hamilton flow determined by the hamiltonian $H$ ,that is, the characteristic curve satisfying the differential equation:
\begin{eqnarray}
\frac{dX(\tau;t,x,v)}{d\tau} = V(\tau;t,x,v) \ \ \ , \ \ \ \frac{dV(\tau;t,x,v)}{d\tau}= -\nabla_x \Phi(X(\tau;t,x,v)),\label{characteristics}
\end{eqnarray}
with $[X(t;t,x,v),V(t;t,x,v)]=[x,v]$.
Clearly the hamiltonian is constant along the characteristics, i.e. \\
$H(X(\tau;t,x,v),V(\tau;t,x,v)) = H(x,v)$ for all $\tau$. Therefore we have an equality $\frac{1}{2}|V(s)|^2 + \Phi(X(s)) = \frac{1}{2}|v|^2 + \Phi(x)$ and further we have
\begin{equation*}
|V(s)|= \sqrt{|v|^2 + 2\Phi(x)-2\Phi(X(s))} \leq \sqrt{|v|^2 + 2|\Phi|_{\infty}} \leq |v| + \sqrt{2|\Phi|_{\infty}} \ .
\end{equation*}
On the other hand we have
\begin{eqnarray*}
|V(s)| = \sqrt{|v|^2 + 2\Phi(x) -2\Phi(X(s))} \geq \sqrt{|v|^2 -2 |\Phi|_{\infty}}\geq |v|-\sqrt{2|\Phi|_{\infty}} \ .
\end{eqnarray*}
Hence we know that
\begin{equation}
\Big{|} \ |V(\tau;t,x,v)|-|v| \ \Big{|} \ \leq  \ 2 |\Phi|_{\infty}^{1/2} \ ,\label{boundofV}
\end{equation}
for all $(\tau , t , x , v)$.

We will use the following geometric result of \cite{Asano94} crucially in this paper.
\begin{lemma}(\cite{Asano94})\label{asano}
Assume that  $ \ \Phi\in C^{3}(\mathbb{R}^d)$. Suppose
$
\ \det\left(\frac{dX(s_0;T_0,x_0,v_0)}{dv}\right)=0
$ for some $(s_0;T_0,x_0,v_0)\in \mathbb{R}\times\mathbb{R}\times\mathbb{R}^d \times\mathbb{R}^d$.
Then there exist $\delta>0$ and an open neighborhood $U_0\subset \mathbb{R}^d \times\mathbb{R}^d$ of $(x_0,v_0)$, and a family of Lipschitz continuous functions on $U_0$, $\{\psi_j : U_0 \rightarrow \mathbb{R}\}_{j=1}^d$ and $\psi_j(0,0)=0$ so that
\begin{equation}
\det\left(\frac{dX(s;T_0,x,v)}{dv}\right)=0,\label{zerodet}
\end{equation}
if and only if
\begin{equation}
s = s_0 + \psi_j (x,v),
\end{equation}
in $(s,x,v)\in (s_0-\delta_0,s_0+\delta_0)\times U_0$.
\end{lemma}
Using Lemma \ref{asano} we construct the $\varepsilon$-neighborhood of the set of points $(s,x,v)$ satisfying (\ref{zerodet}).
\begin{lemma}\label{epsilonneigh}
Assume that $\Phi\in C^{3}(\mathbb{T}^d)$ is a periodic function. Fix $T_0>0$ and $N>0$. There are disjoint open interval partitions of the time interval $[0,T_0]$ : $ \mathfrak{D}_{I^1}^1\subset [0,T_0]$ for $i_1 \in \{1,2,\cdots,M^1\}$ and disjoint open box partitions of the periodic box $\mathbb{T}^d$ : $ \mathfrak{D}_{i^2}^2 \subset \mathbb{T}^d$ for multi-index $I^2=(i_1^2, i_2^2,\cdots,i_d^2)\in\{1,2,\cdots,M^2\}^d$ ; and disjoint open box partitions of $\{v\in\mathbb{R}^d : v_i \in [-4N,4N] \ \text{for all } \ i=1,2\cdots,d\}$ :  $ \mathfrak{D}_{I^3}^3$ for $I^3=(i_1^3, i_2^3,\cdots,i_d^3)\in\{1,2,\cdots,M^3\}^d$. For each $i^1, I^2$ and $I^3$ we have $t_{j,i^1,I^2,I^3}\in \mathfrak{D}_{i^1}^1$ for $j=1,2,...,d$ so that
\begin{equation*}
\bigg\{ s\in \mathfrak{D}_{i_1}^1  :
\det \left( \frac{dX(s;T_0,x,v)}{dv}\right)=0
\bigg\} \ \subset \
\bigcup_{j=1}^d\bigg\{ s\in \left(t_{j,i^1,I^2,I^3}-\frac{\varepsilon}{4M^1} \ , \ t_{j,i^1,I^2,I^3}+\frac{\varepsilon}{4M^1}\right)
\bigg\} \ ,
\end{equation*}
for all $(x,v)\in \mathfrak{D}_{I^2}^2 \times \mathfrak{D}_{I^3}^3$  and
\begin{equation}
\det\left(\frac{dX(s;T_0,x,v)}{dv}\right)  > \delta_* \ \ \ \text{for} \ \ s\notin \bigcup_{j=1}^d \left(t_{j,i^1,I^2,I^3}-\frac{\varepsilon}{4M^1} \ , \ t_{j,i^1,I^2,I^3}+\frac{\varepsilon}{4M^1}\right),
\end{equation}
if $(s,x,v)\in \mathfrak{D}_{i^1}^1 \times \mathfrak{D}_{I^2}^2 \times \mathfrak{D}_{I^3}^3$ for all $i^1, I^2$ and $I^3$.
\end{lemma}
\begin{proof}
Choose $(t_0,x_0,v_0)\in [0,T_0]\times\mathbb{T}^d \times \{ v\in \mathbb{R}^d : v_i \in [-4N,4N], \ \ \text{for} \ i=1,2,...,d\}$.
\newline \underline{First Case} : If
\begin{equation}
\det \left( \frac{dX(t_0;T_0,x_0,v_0)}{ dv}\right) \neq 0, \nonumber
\end{equation}
then there exist positive numbers $\{\tau^0; \xi^0_{1},\cdots,\xi^0_{d}; \eta^0_{1}, \cdots \eta^0_{d}\}$ such that
\begin{eqnarray}
\det \left( \frac{dX(t;T_0,x,v)}{ dv}\right) \neq 0, \ \
\text{for all} \ \ (t;x_1,\cdots,x_d; v_1,\cdots,v_d)\in (t_0 -\tau^0, t_0+\tau^0)\times \ \ \ \ \ \ \ \ \ \ \ \ \ \ \ \ \ \ \ \nonumber\\
\times\big\{((x_0)_1-\xi^0_1 , (x_0)_1+\xi^0_1)\times\cdots\times ((x_0)_d-\xi^0_d , (x_0)_d+\xi^0_d)\big\} \times \big\{((v_0)_1 -\eta^0_1 , (v_0)_1 +\eta^0_1 ) \times \cdots ((v_0)_d -\eta^0_d , (v_0)_d +\eta^0_d )\big\}.\nonumber
\end{eqnarray}
\newline \underline{Second Case} : If
\begin{equation}
\det \left( \frac{dX(t_0;T_0,x_0,v_0)}{ dv}\right) = 0, \nonumber
\end{equation}
then by Lemma \ref{asano}, there exist positive numbers $\{\tau^0; \xi^0_{1},\cdots,\xi^0_{d}; \eta^0_{1}, \cdots \eta^0_{d}\}$ and Lipschitz functions $\psi^0_j$ defined on
\begin{eqnarray*}
\big\{((x_0)_1-\xi^0_1 , (x_0)_1+\xi^0_1)\times\cdots\times ((x_0)_d-\xi^0_d , (x_0)_d+\xi^0_d)\big\} \times \big\{((v_0)_1 -\eta^0_1 , (v_0)_1 +\eta^0_1 ) \times \cdots ((v_0)_d -\eta^0_d , (v_0)_d +\eta^0_d )\big\}
\end{eqnarray*}
and $\psi^0_j(0,0)=0$ so that, for $(t;x_1,\cdots,x_d; v_1,\cdots,v_d)\in (t_0 -\tau^0, t_0+\tau^0)\times\big\{((x_0)_1-\xi^0_1 , (x_0)_1+\xi^0_1)\times\cdots\times ((x_0)_d-\xi^0_d , (x_0)_d+\xi^0_d)\big\} \times \big\{((v_0)_1 -\eta^0_1 , (v_0)_1 +\eta^0_1 ) \times \cdots \times((v_0)_d -\eta^0_d , (v_0)_d +\eta^0_d )\big\},$
\begin{equation}
\det\left( \frac{dX(t;T_0,x,v)}{dv}\right)=0 \ , \label{nonva}
\end{equation}
if and only if
\begin{equation}
t=t_0 + \psi_j^0(x_1,\cdots,x_d;v_1,\cdots v_d) \ \ \text{for some } \ j=1,2,\cdots,d. \label{nonja}
\end{equation}

Notice that from the first and second case, we obtain an open covering of $[0,T_0]\times\mathbb{T}^d\times\{v\in\mathbb{R}^d : v_i \in [-4N,4N]\}$. Since $[0,T_0]\times \mathbb{T}^d \times \{v\in \mathbb{R}^d : v_i \in [-4N,4N]\}$ is a compact set, we can choose finite points $(t_i;(x_i)_1,\cdots,(x_i)_d;(v_i)_1,\cdots(v_i)_d)$ and positive numbers $\{ \tau^i;\xi^i_1,\cdots,\xi^i_d ;\eta^i_1,\cdots,\eta^i_d\}$ for finite index $i \ $'s. Notice that
\begin{eqnarray*}
(t_i -\tau^i, t_i + \tau^i) &\times& \Big\{ \left( \  (x_i)_1 -\xi_1^i \ , \ (x_i)_1 +\xi_1^i \ \right) \times \cdots \times \left( \  (x_i)_d -\xi_d^i \ , \ (x_i)_d +\xi_d^i \ \right)
\Big\}\\
&\times& \Big\{ \left( \  (v_i)_1 -\eta_1^i \ , \ (v_i)_1 +\eta_1^i \ \right) \times \cdots \times \left( \  (v_i)_d -\eta_d^i \ , \ (v_i)_d +\eta_d^i \ \right)
\Big\},
\end{eqnarray*}
forms an open covering for finite $i \ $'s. Define a refined grid via relabeling as
\begin{eqnarray}
\{ \ 0=a_1 < a_2< \cdots < a_{M_1}=T_0 \ \} &=& \{ \ t_i\pm\tau^i \ \ \text{for all} \ \ i \ 's \ \},\\
\{ \ -1 =\tilde{b}_1 < \tilde{b}_2 <\cdots < \tilde{b}_{\tilde{M}_2} =1 \ \} &=&
\{ \ (x_i)_k \pm \xi^i_k \ \ \text{for all} \ \ i \ 's, \ \ k=1,2,\cdots, d \ \},\label{G1}\\
\{ \ -4N =\tilde{c}_1 < \tilde{c}_2 <\cdots < \tilde{c}_{\tilde{M}_3} = 4N \ \} &=&
\{ \ (v_i)_k \pm \eta^i_k \ \ \text{for all} \ \ i \ 's, \ \ k=1,2,\cdots, d \ \}.\label{G2}
\end{eqnarray}
We use the index $i^1 =1,2,\cdots, M^1$ and $\tilde{i}^2_k =1,2,\cdots, \tilde{M}^2$ and $\tilde{i}^3_k =1,2,\cdots , \tilde{M}^3$ for all $k=1,2,\cdots, d$, and multi-index $\tilde{I}^2 = (\tilde{i}^2_1, \tilde{i}^2_2, \cdots, \tilde{i}^2_d)\in \{1,2,\cdots, \tilde{M}^2\}^d$ and $\tilde{I}^3 = (\tilde{i}^3_1, \tilde{i}^3_2, \cdots, \tilde{i}^3_d)\in \{1,2,\cdots, \tilde{M}^3\}^d$. Notice that for each $(j,i^1,\tilde{I}^2,\tilde{I}^3)$ we have $t_{j,i^1,\tilde{I}^2,\tilde{I}^3} \in (a_{i^1},a_{i^1 +1})$ and a Lipschitz function $\psi_{j,i^1,\tilde{I}^2,\tilde{I}^3}$. Using the Lipschitz continuity, we choose a constant $C>0$ so that
\begin{equation*}
| \ \psi_{j,i^1,\tilde{I}^2,\tilde{I}^3}(x,v)- \psi_{j,i^1,\tilde{I}^2,\tilde{I}^3}(\bar{x},\bar{v}) \ | \ \leq \ C \ | \ (x,v)-(\bar{x},\bar{v}) \ | \ ,
\end{equation*}
for all $j, i^1, \tilde{I}^2, \tilde{I}^3$ which are finite indices or finite multi-indices. From the above inequality we have
\begin{equation}
| \ \psi_{j,i^1,\tilde{I}^2,\tilde{I}^3}(x,v)- \psi_{j,i^1,\tilde{I}^2,\tilde{I}^3}(\bar{x},\bar{v}) \ | \ \leq \ \frac{\varepsilon}{8M^1} \ ,\label{111}
\end{equation}
for $|(x,v)-(\bar{x},\bar{v})|\leq \frac{\varepsilon}{8M^1 C}$. Therefore we further refine the grid of (\ref{G1}) and (\ref{G2}) as (if necessary we may put more points to make the grid finer)
\begin{eqnarray*}
\{-1 =b_1 < b_2 <\cdots < b_{{M}^2} =1 \} &\supset& \{-1 =\tilde{b}_1 < \tilde{b}_2 <\cdots < \tilde{b}_{\tilde{M}^2} =1 \}
,\\
\{-4N ={c}_1 < {c}_2 <\cdots < {c}_{{M}^3} = 4N \} &\supset&\{-4N =\tilde{c}_1 < \tilde{c}_2 <\cdots < \tilde{c}_{\tilde{M}^3} = 4N \},
\end{eqnarray*}
and denote multi-indices $I^2 = (i^2_1, i^2_2, \cdots , i^2_d) \in \{1,2,\cdots, M^2\}^d$ and $I^3 = (i^3_1, i^3_2, \cdots , i^3_d) \in \{1,2,\cdots, M^3\}^d$ and define
\begin{eqnarray}
&&\mathfrak{D}_{i^1}^1 \equiv (a_{i^1},a_{i^1 + 1 }),\\
&&\mathfrak{D}_{I^2}^2 =\mathfrak{D}_{i^2_1,\cdots,i^2_d}^2 \equiv (b_{i^2_1}, b_{i^2_1+1}) \times\cdots \times(b_{i^2_d}, b_{i^2_d+1}),\\
&&\mathfrak{D}_{I^3}^3 =\mathfrak{D}_{i^3_1,\cdots,i^3_d}^3 \equiv (c_{i^3_1}, b_{i^3_1+1}) \times\cdots \times(c_{i^3_d}, c_{i^3_d+1}),
\end{eqnarray}
so that $|b_{i^2_1+1}-b_{i^2_1}|+ \cdots +|b_{i^2_d+1}-b_{i^2_d}|+|c_{i^3_1+1}-c_{i^3_1}|+ \cdots +|c_{i^3_d+1}-c_{i^3_d}|\leq \frac{\varepsilon}{8M^1 C}$ and (\ref{111}) is valid for all $(x,v), (\bar{x},\bar{v}) \in \mathfrak{D}_{I^2}^2 \times \mathfrak{D}_{I^3}^3$. From (\ref{nonja}), for each $j,i^1,I^2,I^3$ there exist $t_{j,i^1,I^2,I^3}\in \mathfrak{D}^1_{i^1}$ so that
\begin{eqnarray}
t_{j,i^1,I^2,I^3} + \psi_{j,i^1,{I}^2,{I}^3}(x,v) \ \in \ \bigcup_{j=1}^d( \ t_{j,i^1,I^2,I^3} -\frac{\varepsilon}{2M^1} \ , t_{j,i^1,I^2,I^3} + \frac{\varepsilon}{2M^1} \ ) \ , \ \ \text{for all} \ \ \ (x,v)\in \mathfrak{D}_{I^2}^2 \times \mathfrak{D}_{I^3}^3.\nonumber
\end{eqnarray}
Therefore $$\det\left( \frac{dX(t;T_0,x,v)}{dv}\right)\neq0 \ ,$$ holds for
$
t\in\mathfrak{D}_{i^1}^1 \big{\backslash} ( \ t_{j,i^1,I^2,I^3} -\frac{\varepsilon}{2M^1} \ , t_{j,i^1,I^2,I^3} + \frac{\varepsilon}{2M^1} \ ) \ \ \text{and for} \ \ (x,v)\in\mathfrak{D}_{I^2}^2 \times \mathfrak{D}_{I^3}^3 \ . \nonumber
$
Notice that $\det\left( \frac{dX(t;T_0,x,v)}{dv}\right)$ is continuous and non-zero on a compact set
\begin{eqnarray}
\Big\{ \ \overline{\mathfrak{D}^1_{i^1}} \ \bigg{\backslash} \ \bigcup_{j=1}^d( \ t_{j,i^1,I^2,I^3} -\frac{\varepsilon}{4M^1} \ , t_{j,i^1,I^2,I^3} + \frac{\varepsilon}{4M^1} \ ) \ \Big\} \ \times \ \overline{\mathfrak{D}^2_{I^2}} \ \times \ \overline{\mathfrak{D}^3_{I^3}} \ .\label{compactset}
\end{eqnarray}
Hence there exists positive number $\delta_{i^1,I^2,I^3}>0$ so that $$\det\left(\frac{dX(t;T_0,x,v)}{dv}\right) > \delta_{i^1,I^2,I^3} \ , $$ on the set (\ref{compactset}). Define $\delta_* = \min_{i^1, I^2, I^3} \delta_{i^1,I^2,I^3}>0$ where $i^1, I^2, I^3$ are finite indices. This proves the Lemma.
\end{proof}
\section{Linear $L^{2}$ Decay}
In this section, we will show the $L^2-$decay of the solution of the linear Boltzmann equation :
\begin{equation}
\partial_t f + v\cdot \nabla_x f -\nabla\Phi\cdot\nabla_v f + e^{-\Phi(x)}Lf = 0 \ \ , \ \ f(0,x,v) = f_0 (x,v),\label{linearizedBE}
\end{equation}
assuming the conservation of mass (\ref{masscons}) and energy (\ref{energycons}) and momentum for degenerate $\{v_1,\cdots,v_n\}$ (\ref{momentumcons}) with $(M_0,E_0,\mathbf{J}_0)=(0,0,\mathbf{0})\in\mathbb{R}\times\mathbb{R}\times\mathbb{R}^n$. Here, the number $n$ in the degenerate $\{v_1,\cdots,v_n\}$ is a dimension of the degenerate subspace of $\nabla\Phi$ in (\ref{n}) where, in general, the degenerate subspace of $\nabla\Phi$ is a zero space $\{0\}$ and $n=0$. Therefore, in this section, we are showing actually the linear $L^2-$decay only with the conservation of mass and energy for the generic external potential $\Phi$ in the periodic box $\mathbb{T}^d$.

For notational simplicity, we use $\langle\cdot,\cdot\rangle$ to denote the standard $L^2-$inner product in $\mathbb{T}^d \times\mathbb{R}^d$. We also define
\begin{equation*}
\langle g_1, g_2 \rangle_{\nu} \equiv \langle \nu(v) g_1, g_2 \rangle.
\end{equation*}
We shall use $||\cdot||$ and $||\cdot||_{\nu}$ to denote their corresponding $L^2$ norms.
\begin{theorem}
\label{L2decay}
Assume that the external potential $\Phi$ is a periodic $C^3-$function on $\mathbb{T}^d$ and $\Phi=\Phi(x_{n+1},\cdots,x_d)$. Let $f(t,x,v)\in L^2$ be the (unique) solution to the linear Boltzmann equation (\ref{linearizedBE}). Assume that $f$ satisfies the conservations of mass (\ref{masscons}) and energy (\ref{energycons}) and momentum for degenerate $\{v_1,\cdots,v_n\}$ (\ref{momentumcons}) with $(M_0,E_0,\mathbf{J}_0)=(0,0,\mathbf{0})\in\mathbb{R}\times\mathbb{R}\times\mathbb{R}^n$. Then there exists $\lambda > 0$ and $C>0$ such that
\begin{equation*}
\sup_{0\leq t\leq \infty}e^{\lambda t} || f(t)|| \ \leq \ C || f(0)|| \ .
\end{equation*}
\end{theorem}
\begin{proof}
We will show that Proposition \ref{Positivity} implies Theorem \ref{L2decay}, following the proof of Theorem 5 in \cite{Guo08}. Assume Proposition \ref{Positivity} is true. Let $0\leq N \leq t \leq N+1 \ $, $N$ being an integer. We split $[0,t]= [0,N]\cup [N,t]$. First we establish the $L^2$ energy estimate for any solution $f$ to the linear Boltzmann equation (\ref{linearizedBE}) on the time interval $[N,t]$ as
\begin{equation}
||f(t)||^2 + 2\int_N^t  \int_{\mathbb{T}^d} e^{-\Phi(x)} \int_{\mathbb{R}^d} Lf\cdot f \ dv \ dx \ ds
= ||f(N)||^2.\label{energyESNt}
\end{equation}
From (\ref{linearizedBE}), we have the equation for $e^{\lambda t}f(t)$ :
\begin{equation*}
\{\partial_t + v\cdot \nabla_x -\nabla_x \Phi \cdot \nabla_v + e^{-\Phi(x)}L \}\{e^{\lambda t}f\} -\lambda e^{\lambda t}f = 0.\label{elambdatBE}
\end{equation*}
For the time interval $[0,N]$, we multiply the above equation by $e^{\lambda t} f$ to derive the $L^2-$energy estimate on the time interval $[0,N]$ :
\begin{eqnarray*}
e^{2\lambda N}||f(N)||^2 + 2\int_0^N e^{2\lambda s} \int_{\mathbb{T}^d} e^{-\Phi(x)} \int_{\mathbb{R}^d} Lf\cdot f dv dx ds -\lambda \int_0^N e^{2\lambda s}||f(s)||^2 ds
= ||f(0)||^2.
\end{eqnarray*}
Divide the time interval $[0,N]$ into $\bigcup_{k=0}^{N-1} [k,k+1)$ and define $f_k (s,x,v)\equiv f(k+s ,x,v)$ for $k=0,1,2,...,N-1$. Then we can rewrite the above equation as
\begin{eqnarray}
e^{2\lambda N}|| f(N)||^2 + \underbrace{\sum_{k=0}^{N-1} \int_0^1  e^{2\lambda\{k+s\}} \int_{\mathbb{T}^d} e^{-\Phi(x)} \int_{\mathbb{R}^d}Lf_k (s)\cdot f_k (s) \ dv dx ds }_{\mathbf{(A)}}
=||f(0)||^2 + \underbrace{\lambda \sum_{k=0}^{N-1} \int_0^1 e^{2\lambda \{ k+s \}}||f_k(s)||^2 ds}_{\mathbf{(B)}}.\label{eneryESN0}
\end{eqnarray}
Notice that $f_k (k+s,x,v)$ satisfies the linear Boltzmann equation (\ref{linearizedBE}) in the time interval $[0,1]$. By Proposition \ref{Positivity}, we have a lower bound for $\mathbf{(A)}$ as
\begin{eqnarray*}
\mathbf{(A)} \geq\sum_{k=0}^{N-1} e^{2\lambda k} e^{-|\Phi|_{\infty}} \int_0^1 \langle Lf_k(s), f_k(s)\rangle ds \geq \sum_{k=0}^{N-1} e^{2\lambda k} e^{-|\Phi|_{\infty}} M \int_0^1 ||f_k(s)||_{\nu}^2 ds\geq \nu_0 M e^{-|\Phi|_{\infty}} \sum_{k=0}^{N-1}  e^{2\lambda k}  \int_0^1 ||f_k(s)||^2 ds.
\end{eqnarray*}
Using the fact $e^{2\lambda (k+s)} \leq e^{2\lambda} e^{2\lambda k}$ for $s\in [0,1]$, we have $\mathbf{(B)} \leq \lambda e^{2\lambda} \sum_{k=0}^{N-1} \int_0^1 e^{2\lambda k} ||f_k(s)||^2 ds$. Thus, for sufficiently small $\lambda>0$, we have a positive lower bound of $\mathbf{(A)}-\mathbf{(B)}$ as
\begin{equation*}
\mathbf{(A)}-\mathbf{(B)} \geq (\nu_0M e^{-|\Phi|_{\infty}} -\lambda e^{2\lambda}) \sum_{k=0}^{N-1} \int_0^1 e^{2\lambda k} ||f_k(s)||^2 ds \geq 0.
\end{equation*}
Therefore from (\ref{eneryESN0}), we have
\begin{equation*}
e^{2\lambda N} ||f(N)||^2 \ \leq \ ||f(0)||^2.\label{expdecayN}
\end{equation*}
Further we can choose $\lambda>0$ small so that $e^{2\lambda (t-N)} \leq 2$ for all $t\in[N,N+1]$. Hence, multiply (\ref{energyESNt}) by $e^{2\lambda t}$ and combine with the above inequality to conclude
\begin{eqnarray*}
e^{2\lambda t} ||f(t)||^2 \ \leq \ e^{2\lambda t}||f(N)||^2 \ \leq \ e^{2\lambda\{t-N\}} ||f_0 ||^2 \ \leq \ 2||f_0||^2.
\end{eqnarray*}
\end{proof}
\begin{proposition}\label{Positivity}
Assume that the external potential $\Phi$ is a periodic $C^3$-function on $\mathbb{T}^d$ and $\Phi=\Phi(x_{n+1},...,x_d)$.
Let $f(t,x,v)$ be any solution to the linear Boltzmann equation (\ref{linearizedBE}) satisfying the conservations of mass (\ref{masscons}) and energy (\ref{energycons}) and momentum for degenerate $\{v_1,\cdots,v_n\}$ (\ref{momentumcons}) with $(M_0,E_0,\mathbf{J}_0)=(0,0,\mathbf{0})\in\mathbb{R}\times\mathbb{R}\times\mathbb{R}^n$. Then there exists $M>0$ such that $f$ satisfies
\begin{equation}
\int_0^1 \langle Lf(s), f(s)\rangle ds \ \geq \ M \int_0^1 ||f(s)||_{\nu}^2 ds. \label{positivity}
\end{equation}
\end{proposition}
\begin{proof}
We prove the Proposition by the contradiction argument. If the inequality of (\ref{positivity}) is not true, then a sequence of solutions $f_k (t,x,v)$ (not identically zero) to (\ref{linearizedBE}) exists so that
\begin{equation*}
\int_0^1 \langle Lf_k(s), f_k(s)\rangle ds \leq \frac{1}{k} \int_0^1 ||f_k(s)||_{\nu}^2 ds.
\end{equation*}
Equivalently, in terms of normalization $Z_k (t,x,v) = \frac{f_k (t,x,v)}{\sqrt{\int_0^1 ||f_k (s)||_{\nu}^2 ds}}$, we have
\begin{equation}
\int_0^1 \langle LZ_k(s) ,Z_k(s) \rangle ds \leq \frac{1}{k},\label{44}
\end{equation}
and from $\nu(v)\geq \nu_0, $ we have
$
\nu_0 \int_0^1 ||Z_k(s)||^2 ds \leq \int_0^1 ||Z_k(s)||_{\nu}^2 ds =1.\nonumber
$
Due to the weak compactness in $L^2$ space, there exists $Z(t,x,v)$ with $\int_0^1||Z(s)||_{\nu} ds \leq 1$ such that
\begin{equation}
Z_k \rightharpoonup Z  \ \ \text{weakly in} \ \ \int_0^1 ||\cdot||_{\nu}^2 ds \ \ \text{and} \ \ \int_0^1 ||\cdot||^2 ds.\label{weakconv}
\end{equation}
On the other hand, since $f_k (t,x,v)$ solves (\ref{linearizedBE}), $Z_{k}(t,x,v)$ satisfies the same equation
\begin{equation}
\{\partial_t +v\cdot\nabla_x -\nabla\Phi(x)\cdot\nabla_v \}Z_k (t,x,v) + e^{-\Phi(x)}LZ_k(t,x,v)=0,\label{eqntZ_k}
\end{equation}
and hence $Z_k(t,x,v)$ satisfies same conservation laws (\ref{masscons}), (\ref{energycons}), and (\ref{momentumcons}) with $(M_0, E_0, \mathbf{J}_0)=(0, 0, \mathbf{0})\in\mathbb{R}\times\mathbb{R}\times\mathbb{R}^n$ as $f_k(t,x,v)$ does. Using the weak convergence in (\ref{weakconv}), we conclude that for almost every $t\in [0,1]$, the limiting function $Z(t,x,v)$ also satisfies
\begin{eqnarray}
&&\iint_{\mathbb{T}^d \times\mathbb{R}^d} Z(t,x,v) \sqrt{\mu_E(x,v)} dv dx \ = \ 0 \ ,\label{conserlawZ_k_1}\\
&&\iint_{\mathbb{T}^d \times\mathbb{R}^d} \left(\Phi(x)+\frac{|v|^2}{2}\right)Z(t,x,v) \sqrt{\mu_E(x,v)} dv dx \ = \ 0 \ ,\label{conserlawZ_k_2}\\
&&\iint_{\mathbb{T}^d \times\mathbb{R}^d} (v_1,\cdots,v_n)^T \ Z(t,x,v) \sqrt{\mu_E(x,v)} dv dx \ = \ 0\label{conserlawZ_k_3} \ .
\end{eqnarray}
\newline\textbf{Step 1} : $\mathbf{P}Z_k \rightharpoonup \mathbf{P}Z \ \ \text{weakly in} \ \ \int_0^1 ||\cdot||^2_{\nu} ds \ \ \text{and} \ \ \int_0^1 ||\cdot||^2 ds.$\\
It suffices to show that
\begin{equation}
\int_0^1 \langle \mathbf{P}Z_k, \phi_1(s,x) \phi_2(v) \rangle_{\nu}  ds \rightarrow \int_0^1 \langle \mathbf{P}Z, \phi_1(s,x) \phi_2(v) \rangle_{\nu}  ds,\label{weak}
\end{equation}
for any smooth functions $\phi_1(s,x)$ on $[0,1]\times\mathbb{T}^d$ and $\phi_2(v)$ on $\mathbb{R}^d$ with compact supports.

Temporally denote $(m_0,m_1,\cdots,m_d,m_{d+1})\equiv(1,v_1,\cdots,v_d,|v|^2)$. Then the left-hand side of (\ref{weak}) is written componentwisely as
\begin{eqnarray}
\int_0^1\int_{\mathbb{T}^d} \int_{\mathbb{R}^d}\left(
\int_{\mathbb{R}^d} m_i(u) \sqrt{\mu(u)} \ Z_k(s,x,u) du \right) m_i(v)\sqrt{\mu(v)} \ \nu(v) \phi_1(s,x) \phi_2(v) dv dx ds\nonumber\\
=\int_{\mathbb{R}^d} \underline{\left(
\int_0^1 \int_{\mathbb{T}^d} \int_{\mathbb{R}^d}
m_i \sqrt{\mu} \ Z_k \phi_1
du dx ds
\right) m_i(v) \sqrt{\mu(v)}\nu(v)\phi_2(v)}
dv.\label{above1}
\end{eqnarray}
Notice that the underlined integrand of (\ref{above1}) is bounded by $L^1$ function uniformly in $k $ 's , i.e.
\begin{equation}
\left(\int_0^1 || m_i \phi_1 \sqrt{\mu}||_{L^2_{x,v}}^2 ds\right)^{\frac{1}{2}}
\left(
\int_0^1 ||Z_k(s)||_{L^2_{x,v}}^2 ds\right)^{\frac{1}{2}} m_i(v) \nu(v)\phi_2(v)\sqrt{\mu(v)} \leq C  m_i(v) \nu(v)\phi_2(v)\sqrt{\mu(v)} \in L^1(\mathbb{R}^d),\nonumber
\end{equation}
and the underlined integrand of (\ref{above1}) converges to
$
\left(
\int_0^1 \int_{\mathbb{T}^d} \int_{\mathbb{R}^d}
m_i \sqrt{\mu} \ Z \phi_1
du dx ds
\right) m_i(v) \sqrt{\mu(v)}\nu(v)\phi_2(v)\nonumber
$ for almost every $v\in\mathbb{R}^d$. By the Lebesque convergence theorem, we conclude the convergence of (\ref{weak}).\\
\newline\textbf{Step 2} : $K(Z_k) \rightarrow K(Z) \ \ \ \text{in} \ \ \ L^2 ([0,1]\times\mathbb{T}^d \times \mathbb{R}^d).$ \\
The proof of this step is same as Step 1 in the proof of Lemma 4.1 in \cite{GuoVPB} page 1124. Denote $k(u,v)$ as the kernel of the operator $K$ and define it's approximation $k_m(u,v) \equiv k(u,v) \mathbf{1}_{\{ (u,v) : |u-v|\geq \frac{1}{m}, |v|\leq m\}}$. Since $k_m \in L^2 (\mathbb{R}^d \times\mathbb{R}^d)$ we can choose smooth functions with compact supports to satisfy
\begin{equation*}
\kappa_{\varepsilon}(u,v) = \kappa_1(u) \kappa_2(v) \ \ \ \text{such that} \ \ \  ||k_m -\kappa_{\varepsilon}||^2 \leq \varepsilon.
\end{equation*}
In order to prove this step, we only need to change (4.8) in \cite{GuoVPB} by
\begin{eqnarray}
&&[\partial_t + v\cdot\nabla_x -\nabla_x\Phi(x)\cdot\nabla_v] \left\{ \kappa_1(v) \chi(t,x) Z_k(t,x,v)\right\} \nonumber\\
&& \ \ \ \ \ \ \ \ \ \ \ \ \ \ \ \ \ \ \ \ \ \ \ = [\partial_t + v\cdot\nabla_x -\nabla_x\Phi(x)\cdot\nabla_v] \{ \kappa_1(v)\chi(t,x)\}Z_k(t,x,v),\label{eq1}
\end{eqnarray}
where $\chi(t,x)$ is a smooth cut-off function in $(0,1)\times\mathbb{R}^d$ such that $\chi(t,x)\equiv 1$ in $[\varepsilon,1-\varepsilon]\times\mathbb{T}^d$. It is strainghforward to verify that the right hand side of (\ref{eq1}) is uniformly bounded in $L^2([0,1]\times\mathbb{R}^d \times\mathbb{R}^d)$. From the velocity average lemma(\cite{BJ},\cite{GLPS}), for some $\alpha>0$,
$$
\int_{\mathbb{R}^d} \chi(t,x) \kappa_1(u) Z_k (t,x,u) du \ \in \ H^{\alpha}([0,1]\times\mathbb{R}^d).
$$
It follows that, up to a subsequence,
$$
\int_{\mathbb{R}^d} \kappa_1(u) Z_k(s,x,u) du \rightarrow \int_{\mathbb{R}^d} \kappa_1(u)Z(s,x,v) du,
$$
strongly in $L^2([\varepsilon,1-\varepsilon]\times\mathbb{T}^d)$ and this suffices to prove \textbf{Step 2}. For detail, see the proof of Lemma 4.1 in \cite{GuoVPB}.\\
\newline\textbf{Step 3} : $Z(t,x,v) = \{a(t,x) + v\cdot \mathbf{b}(t,x) + |v|^2 c(t,x) \}\sqrt{\mu_E(x,v)} \ \neq \ 0$.\\
From the definitions of $L=\nu-K$ and $Z_k$, we have
\begin{eqnarray}
1- \int_0^1 \langle KZ_k(s),Z_k(s) \rangle ds=
\int_0^1 || Z_k(s)||_{\nu} ds - \int_0^1 \langle KZ_k(s),Z_k(s) \rangle ds
= \int_0^1 \langle LZ_k(s), Z_k(s) \rangle ds .\label{sharp}
\end{eqnarray}
Combining the strong convergence in \textbf{Step 2} with the weak convergence of $Z_k$, we conclude
\begin{equation*}
\lim_{k\rightarrow\infty}\int_0^1\langle KZ_k(s), Z_k(s)\rangle = \int_0^1\langle KZ(s), Z(s)\rangle.
\end{equation*}
Combining the above relation with (\ref{sharp}) and $0\leq \int_0^1 \langle LZ_k(s),Z_k(s)\rangle ds \leq \frac{1}{k}$, we conclude
\begin{equation*}
0=1-\int_0^1 \langle KZ(s),Z(s)\rangle ds. \label{zero}
\end{equation*}
Using the above equation, with the non-negativity of $L$ and $\int_0^1 ||Z(s)||_{\nu} ds \leq 1$, we have
\begin{eqnarray*}
0\leq \int_0^1 \langle LZ(s), Z(s) \rangle ds = \int_0^1 ||Z(s)||_{\nu} ds - \int_0^1 \langle KZ(s),Z(s)\rangle ds
\leq 1- \int_0^1 \langle KZ(s),Z(s)\rangle ds = 0,
\end{eqnarray*}
to conclude
\begin{equation}
\int_0^1 \langle LZ(s), Z(s) \rangle ds =0 \ , \ \  \ \ \ \int_0^1 ||Z(s)||_{\nu} ds =1. \label{zeroo}
\end{equation}
Now use the standard property of $L$ :
\begin{equation}
\int_{\mathbb{R}^d} LZ(s,x,v) \cdot Z(s,x,v)dv \geq  C\int_{\mathbb{R}^d} \nu(v)|\{\mathbf{I}-\mathbf{P}\}Z(s,x,v)|^2 dv,\label{sharpsharp}
\end{equation} and combine with (\ref{zeroo}) to conclude
\begin{equation}
\{\mathbf{I}-\mathbf{P}\}Z(t,x,v)=0,\nonumber
\end{equation}
for almost every $(t,x,v)\in [0,1]\times\mathbb{T}^d \times\mathbb{R}^d$. Hence $
Z=\mathbf{P}Z = \{\tilde{a}(t,x)+ v\cdot \mathbf{\tilde{b}}(t,x) +|v|^2 \tilde{c}(t,x)\} \sqrt{\mu(v)}
$ where $[\tilde{a}(t,x),\mathbf{\tilde{b}}(t,x),\tilde{c}(t,x)]$ is linear combinations of $\left[\int Z(t,x,\cdot)\sqrt{\mu}dv, \int v Z(t,x,\cdot)\sqrt{\mu}dv , \int |v|^2Z(t,x,\cdot)\sqrt{\mu}dv\right] $. \\
Define $[a(t,x),\mathbf{b}(t,x),c(t,x)]=[\tilde{a}(t,x),\mathbf{\tilde{b}}(t,x),\tilde{c}(t,x)]\times e^{\frac{\Phi(x)}{2}}$ then we have
\begin{equation}
Z(t,x,v) = \{a(t,x) + v\cdot \mathbf{b}(t,x) + |v|^2 c(t,x) \}\sqrt{\mu_E}\label{Z} \ \ .
\end{equation}
From (\ref{zeroo}) we conclude that $Z(t,x,v)$ is not identically zero.\\
\newline\textbf{Step 4} : $Z\equiv 0$\\
This leads to a contradiction to \textbf{Step 3}. Notice that from (\ref{44}) and (\ref{sharpsharp}), we have
$$
\int_0^1 ||\{\mathbf{I}-\mathbf{P}\}Z_k(s)||_{\nu}^2 ds \ \leq \ \frac{1}{C} \int_0^1 \langle LZ_k(s), Z_k(s)\rangle ds \ \rightarrow \ 0,
$$
to conclude $\{\mathbf{I}-\mathbf{P}\}Z_k \rightarrow 0$ strongly in $\int_0^1 ||\cdot||_{\nu} ds$.
From $LZ_k = L\{\mathbf{I}-\mathbf{P}\}Z_k$, we have $\int_0^1 \langle L\{\mathbf{I}-\mathbf{P}\}Z_k, \varphi\rangle ds =\int_0^1 \langle LZ_k, \varphi\rangle ds \rightarrow 0$ for all $\varphi\in C_c^{\infty}([0,1]\times\mathbb{T}^d \times \mathbb{R}^d)$. Hence letting $k\rightarrow \infty$ in (\ref{eqntZ_k}), we have, in the sense of distribution,
\begin{equation*}
\partial_t Z + v\cdot\nabla_x Z -\nabla\Phi(x)\cdot \nabla_v Z =0 \label{eqntZ}.
\end{equation*}
We plug (\ref{Z}) into the above equation and expand as the products of a polynomial in $v_i$  :
\begin{eqnarray}
\{\dot{a} - \nabla_x\Phi\cdot b\}\sqrt{\mu_E} +\{\dot{b}+ \nabla_x a -2c\nabla_x \Phi \}\cdot v\sqrt{\mu_E}
+\sum_i (\dot{c}+\partial_{x_i}b_i)|v_i|^2 \sqrt{\mu_E} \ \ \ \ \ \ \ \ \ \ \ \ \ \ \ \ \ \ \ \nonumber\\
 \ \ \ \ \ \ \ \  \ \ \ \ \ \ \ \ \ \ \ \ \ \  \ \ \ \ \ \ \ \ \ \ \ \ + \sum_{i\neq j} \{\partial_{x_i}b_j\} v_i v_j \sqrt{\mu_E} + \nabla_x c \cdot v |v|^2 \sqrt{\mu_E}=0.\label{eqtnMM}
\end{eqnarray}
Since $\sqrt{\mu}, v_i \sqrt{\mu}, v_j v_k \sqrt{\mu}$ and $v_l v_m v_n \sqrt{\mu}$ are linearly independent, we deduce that in the sense of distributions, all the coefficients on the left hand side of (\ref{eqtnMM}) should be zero. We therefore obtain the \textbf{macroscopic equations} of $a(t,x), \mathbf{b}(t,x)$ and $c(t,x)$, which was introduced in \cite{GuoVPB} :
\begin{eqnarray}
\dot{a}(t,x) - \nabla_x \Phi(x) \cdot \mathbf{b}(t,x) &=& 0\label{ME1}\\
\dot{\mathbf{b}}(t,x) + \nabla_x a(t,x) - 2c(t,x) \nabla_x \Phi(x) &=& 0\label{ME2}\\
\dot{c}(t,x) + \partial_{x_i}b_i &=& 0 \ \ \ \ \text{for all } \ i \label{ME3}\\
\partial_{x_j}b_i + \partial_{x_i}b_j &=& 0 \ \ , \ \ i\neq j\label{ME4}\\
\nabla_x c(t,x) &=&0\label{ME5}
\end{eqnarray}
We obtain the Laplace equation of $b_i (t,x)$
\begin{eqnarray}
\Delta b_i
&=& \sum_{j=1}^d \partial_{j} \partial_{j} b_i
= \sum_{j\neq i} \partial_{j}\underline{\partial_{j}b_i}_{*} + \partial_{i}\underline{\partial_{i}b_i}_{**}
= \sum_{j\neq i} \partial_{j}\underline{(-\partial_{i}b_j)}_{\diamond} + \partial_{i}\underline{(-\dot{c})}_{\diamond\diamond}\nonumber\\
&=& -\partial_i \left(\sum_{j\neq i}\underline{\partial_j b_j}_{**} \right) - \partial_i \dot{c} = -(d-1)\partial_i \underline{(-\dot{c})}_{\diamond\diamond}-\partial_i \dot{c}
=(d-2)\underline{\partial_i \dot{c} \ = \ 0}_{***} \label{b2}
\end{eqnarray}
where we used (\ref{ME4}) for $*=\diamond$ and used (\ref{ME3}) for $**=\diamond\diamond$ and used (\ref{ME5}) for $***$. From (\ref{Z}) we have, for all $i=1,2,..,d,$
\begin{eqnarray*}
\int_{\mathbb{R}^d} v_i Z(t,x,v) dv = \sum_j \int_{\mathbb{R}^d} v_i v_j b_j (t,x) e^{-\frac{|v|^2}{4}}e^{-\frac{\Phi(x)}{2}} dv = b_i (t,x) e^{-\frac{\Phi(x)}{2}} \int_{\mathbb{R}^d} |v_i|^2 e^{-\frac{|v|^2}{4}}dv.
\end{eqnarray*}
Since $Z(t,x,v)$ is periodic in $x$, we conclude that $\mathbf{b}(t,x)$ is also periodic in $x$. Using the periodicity of $\mathbf{b}$, multiply (\ref{b2}) by $b_i$ and integrate to yield
\begin{equation*}
0= \int_{\mathbb{T}^d} \Delta b_i b_i dx = - \int_{\mathbb{T}^d} |\nabla b_i|^2 dx,
\end{equation*}
and
\begin{equation*}
\nabla_x b_i(t,x)=0 \ ,   \ \ \  \ \mathbf{b}(t,x)=\mathbf{b}(t) \ \ \ \ \text{for almost all} \ \ (t,x)  .
\end{equation*}
Combining the above equality with (\ref{ME3}) and (\ref{ME5}), we conclude that
\begin{equation}
c(t,x) =c_0 \ \ \ \ \text{for almost all} \ \ (t,x) .\label{const-c}
\end{equation}
Further integrate (\ref{ME2}) on $\mathbb{T}^d$ to have
\begin{eqnarray*}
0= \int_{\mathbb{T}^d} \ (\ref{ME2}) \ dx=\dot{\mathbf{b}}(t) + \int_{\mathbb{T}^d} \nabla_x a(t,x) dx -2c_0 \int_{\mathbb{T}^d} \nabla_x \Phi(x) dx = \dot{\mathbf{b}}(t),
\end{eqnarray*}
where we used the periocity in $x\in \mathbb{T}^d$ of $a(t,x)$ and $\Phi(x)$. Therefore we conclude $\mathbf{b}(t,x) =\mathbf{b}_0.$
From (\ref{ME1}) and (\ref{ME2}), we have equations of $a(t,x)$:
\begin{eqnarray*}
\nabla_x a(t,x) = 2c_0\nabla_x \Phi(x) \ , \ \ \ \
\dot{a}(t,x) = \mathbf{b}_0 \cdot \nabla_x \Phi(x).
\end{eqnarray*}
From the first equation above, we have $a(t,x)=2c_0 \Phi(x)+ \varphi(t)$ for some $\varphi$. Plugging this formula into the second equation above, we get
\begin{equation*}
\dot{\varphi}(t)=\dot{a}(t,x)=\mathbf{b}_0 \cdot \nabla_x \Phi(x).\label{dd}
\end{equation*}
Since the left hand side of the above equation is a function of $t$ only and the right hand side is a function of $x$ only, we conclude that both of them are constant. In order to show that the both sides are zero actually, we utilize the periodicity of the external potential $\Phi$ : take the integration over $x\in\mathbb{T}^d$ to yield
\begin{equation*}
\int_{\mathbb{T}^d} \mathbf{b}_0 \cdot \nabla_x \Phi(x) dx =  \int_{\mathbb{T}^d}  \nabla_x \cdot\{ \Phi(x)\mathbf{b}_0\} dx= 0.
\end{equation*}
Therefore we conclude that, for all $t\in\mathbb{R}$ and $x\in\mathbb{T}^d$, we have $\varphi(t) \ \equiv \ \varphi_0$ and \begin{eqnarray}
&& \ \ \ \ \ \ \ a(t,x) \ \equiv \ 2c_0 \Phi(x)+\varphi_0 \ , \label{a}\\
&&\mathbf{b}_0 \cdot \nabla_x \Phi(x)  \ \equiv \ 0 \ \ \text{for all } \ x\in\mathbb{T}^d \ .\nonumber
\end{eqnarray}
Recall $\Lambda(\mathbb{T}^d)$, the degenerate subspace of $\nabla\Phi$ in (\ref{v1tovn}). By the definition, we have
\begin{eqnarray}
\mathbf{b}_0 \in \Lambda(\mathbb{T}^d)^{\bot}, \ \ \ \text{and hence} \ \ \ \mathbf{b}_0 = ((\mathbf{b}_0)_1,\cdots, (\mathbf{b}_0)_n,0,\cdots,0).\label{zero-b}
\end{eqnarray}
For simplicity we think $\mathbf{b}_0$ as a vector in $\mathbb{R}^n.$
To sum, we plug (\ref{const-c}), (\ref{a}) and (\ref{zero-b}) into (\ref{Z}) to conclude
\begin{equation}
Z(t,x,v) = \big\{ \varphi_0 +\mathbf{b}_0\cdot(v_1, \cdots,v_n ) + 2c_0 \big( \Phi(x) +\frac{|v|^2}{2} \big) \big\} \sqrt{\mu_E(x,v)}.
\end{equation}
Notice that we obtain the above formula for $Z$ only using the macroscopic equations and periodicity in $x\in\mathbb{T}^d$.

In order to conclude $Z\equiv 0$, we use the conservations of mass (\ref{conserlawZ_k_1}) and energy (\ref{conserlawZ_k_2}) and momentums for degenerate $\{v_1,\cdots,v_n\}$ (\ref{conserlawZ_k_3}) crucially. From the conservation of momentum for degenerate $\{v_1, \ldots,v_n\}$, we have
\begin{equation}
0= \int_0^1 \iint_{\mathbb{T}^d \times\mathbb{R}^d} v_i Z(s,x,v) \sqrt{\mu_E (x,v)} dv dx ds =
(\mathbf{b}_0)_i \int_0^1 \iint_{\mathbb{T}^d \times\mathbb{R}^d} (v_i)^2 \mu_E(x,v) dv dx ds,\label{cruciall}
\end{equation}
for all $i=1,\cdots,n$, so that $\mathbf{b}_0\equiv \mathbf{0} \in \mathbb{R}^n$.

From the conservation of mass for $Z$ in (\ref{conserlawZ_k_1}), we have
\begin{equation}
0= \varphi_0 \iint_{\mathbb{T}^d \times\mathbb{R}^d} \mu_E(x,v) dv dx +
2c_0 \iint_{\mathbb{T}^d \times\mathbb{R}^d} \left(\Phi(x)+\frac{|v|^2}{2}\right)\mu_E(x,v) dv dx,\nonumber
\end{equation}
and from the conservation of energy for $Z$ in (\ref{conserlawZ_k_2}), we have
\begin{equation}
0= \varphi_0 \iint_{\mathbb{T}^d \times\mathbb{R}^d}\left(\Phi(x)+\frac{|v|^2}{2}\right) \mu_E(x,v) dv dx +
2c_0 \iint_{\mathbb{T}^d \times\mathbb{R}^d} \left(\Phi(x)+\frac{|v|^2}{2}\right)^2\mu_E(x,v) dv dx.\nonumber
\end{equation}
Using the notation $\langle\cdot,\cdot\rangle$ for $L^2(\mathbb{T}^d\times\mathbb{R}^3)$ inner product, the above two equations are
\begin{eqnarray*}
\left(\begin{array}{ccc} \big\langle \sqrt{\mu_E},\sqrt{\mu_E} \ \big\rangle & \big\langle (\Phi(x)+{|v|^2}/{2})\sqrt{\mu_E}, \ \sqrt{\mu_E} \ \big\rangle\\
\big\langle (\Phi(x)+{|v|^2}/{2})\sqrt{\mu_E}, \ \sqrt{\mu_E} \ \big\rangle &
\big\langle (\Phi(x)+{|v|^2}/{2})\sqrt{\mu_E}, \ (\Phi(x)+{|v|^2}/{2})\sqrt{\mu_E} \ \big\rangle
\end{array}
\right)
\left(\begin{array}{ccc}
\varphi_0 \\
2c_0
\end{array}\right)
=\left(\begin{array}{ccc}0\\0\end{array}\right).
\end{eqnarray*}
Once we show that the determinant of the above matrix is not zero :
\begin{equation}
\big\langle \sqrt{\mu_E},\sqrt{\mu_E} \ \big\rangle \big\langle (\Phi(x)+{|v|^2}/{2})\sqrt{\mu_E}, \ (\Phi(x)+{|v|^2}/{2})\sqrt{\mu_E} \ \big\rangle
-\big\langle (\Phi(x)+{|v|^2}/{2})\sqrt{\mu_E}, \ \sqrt{\mu_E} \ \big\rangle^2 \neq 0,\label{nonzz}
\end{equation}
then we conclude $\varphi_0=0$ and $c_0=0$ and hence $Z\equiv 0$. From the Cauchy-Schwartz inequality
\begin{equation*}
\big\langle (\Phi(x)+{|v|^2}/{2})\sqrt{\mu_E}, \ \sqrt{\mu_E} \ \big\rangle^2 \lneqq
\big\langle \sqrt{\mu_E},\sqrt{\mu_E} \ \big\rangle \big\langle (\Phi(x)+{|v|^2}/{2})\sqrt{\mu_E}, \ (\Phi(x)+{|v|^2}/{2})\sqrt{\mu_E} \ \big\rangle,
\end{equation*}
if $\sqrt{\mu_E}$ and $(\Phi(x)+{|v|^2}/{2})\sqrt{\mu_E}$ are linearly independent which is obvious. Thus we conclude (\ref{nonzz}).
\end{proof}

\section{Nonlinear $L^{\infty}$ Decay}
In this section, we prove Theorem 1, especially the nonlinear $L^{\infty}-$decay in (\ref{nondecay}). Recall the weight function by $w(x,v)=\{\frac{|v|^2}{2} + \Phi(x)\}^{\beta /2}$ in (\ref{weight}) and define a weighted perturbation by
\begin{equation}
h(t,x,v) = w(x,v) \times \frac{F(t,x,v)-\mu_E(x,v)}{\sqrt{\mu_E(x,v)}}.\label{h}
\end{equation}
Notice that $h(t,x,v)=w(x,v)f(t,x,v)$. Then $h$ satisfies
\begin{eqnarray}
\{\partial_t + v\cdot \nabla_x - \nabla\Phi(x)\cdot \nabla_v\} h(t,x,v) + e^{-\Phi(x)}\nu h(t,x,v) - e^{-\Phi(x)}K_w h
=  e^{-\frac{\Phi(x)}{2}} w \Gamma (\frac{h}{w},\frac{h}{w}),\label{linearBE-h}
\end{eqnarray}
where $L_w h = \nu(v) h - K_w h$ with $K_w h = w K (\frac{h}{w})$ (\cite{Guo08}). Notice that via Lemma 19 in \cite{Guo08}, assuming $\sup_{0\leq s\leq T_0}e^{\lambda s} ||h(s)||_{\infty}$ is small, we only have to show that there exist $\lambda > 0$ and $T_0 >0$ and $C_{T_0}>0$ such that
\begin{equation}
|| h(T_0)||_{\infty} \leq e^{-\lambda T_0}||h_0||_{\infty} + C_{T_0}\int_0^{T_0} ||f(s)||_{L^2}ds, \label{expdecayT}
\end{equation}
in order to show the nonlinear $L^{\infty}-$decay, i.e.
\begin{equation}
\sup_{0\leq t\leq \infty}e^{\lambda t}||h(t)||_{\infty} \leq C ||h_0||_{\infty},\label{expdecay}
\end{equation}
which is equivalent to (\ref{nondecay}). Once we establish (\ref{expdecay}), proving the existence and uniqueness, positivity of the Boltzmann solution $F$ were established in \cite{Guo08}.

For any $(t,x,v)$, integrating along its backward trajectory $\frac{dX(s)}{ds}=V(s), \frac{dV(s)}{ds}=-\nabla\Phi(X(s))$ in (\ref{characteristics}), we express
\begin{eqnarray}
h(t,x,v)&=&e^{-\int_0^t e^{-\Phi(X(\tau))} \nu(V(\tau)) d\tau} h(0,X(0),V(0))\label{1}\\
&+& \int_0^t e^{-\Phi(X(s)) - \int_s^t e^{-\Phi(X(\tau))} \nu(V(\tau)) d\tau
}
K_w h (s,X(s),V(s)) ds\label{2}\\
&+& \int_0^t e^{-\frac{\Phi(X(s))}{2} - \int_s^t e^{-\Phi(X(\tau))} \nu(V(\tau)) d\tau
} w\Gamma\left(\frac{h}{w},\frac{h}{w}\right)(s,X(s),V(s)) ds. \label{3}
\end{eqnarray}
Easily we can control the first line above by
\begin{equation}
(\ref{1})\leq e^{-{t\nu_0 e^{-|\Phi|_{\infty}}}}  ||h(0)||_{L^{\infty}_{x,v}}.\label{one}
\end{equation}
\newline Next we estimate (\ref{3}). We can bound the loss term in (\ref{3}) :
\begin{eqnarray*}
w(X(s),V(s)) \Gamma_- \left( \frac{h}{w} , \frac{h}{w} \right)(s,X(s),V(s))=  \int_{\mathbb{R}^d} \int_{\mathbb{S}^{d-1}} q(\omega, |u-V(s)|) \frac{e^{-\frac{|u|^2}{4}}}{w(X(s),u)} h(s,X(s),u) h(s,X(s),V(s))d\omega du\\
\leq
\left[ \iint q(\omega, |u-V(s)|) e^{-\frac{|u|^2}{4}} w^{-1}(X(s),u) du d\omega
\right] \times || h(s) ||_{L^{\infty}_{x,v}}^2 \leq \nu(V(s)) || h(s) ||_{L^{\infty}_{x,v}}^2, \ \ \ \ \ \ \ \ \ \ \ \ \ \ \
\end{eqnarray*}
and for the gain term
\begin{eqnarray*}
w(X(s),V(s))\Gamma_+ \left( \frac{h}{w} , \frac{h}{w} \right)(s,X(s),V(s))=w(X(s),V(s)) \int_{\mathbb{R}^d}\int_{\mathbb{S}^{d-1}} q(\omega, |u-V(s)|) e^{-\frac{|u|^2}{4}} \frac{h(s,X(s),u^{\prime})}{w(X(s),u^{\prime})} \frac{h(s,X(s),v^{\prime})}{w(X(s),v^{\prime})} d\omega du\\
\leq\Phi(X(s))^{-\beta /2}\int_{\mathbb{R}^d} \int_{\mathbb{S}^{d-1}} q(\omega,|u-V(s)|)e^{-\frac{|u|^2}{4}} d\omega du \times || h(s) ||_{L^{\infty}_{x,v}}^2 \leq \nu(V(s)) || h(s) ||_{L^{\infty}_{x,v}}^2, \ \ \ \ \ \ \ \ \ \ \ \ \ \ \ \ \ \ \ \ \
\end{eqnarray*}
where $v^{\prime}= v^{\prime}(u,V(s)) \ , \ u^{\prime} = u^{\prime}(u,V(s))$ and we used $|u|^2 + |V(s)|^2 = |u^{\prime}|^2 + |v^{\prime}|^2$ so that
\begin{eqnarray*}
w(X(s),u^{\prime}) w(X(s),v^{\prime}) = (\Phi(X(s))+\frac{|u^{\prime}|^2}{2})^{\beta /2} (\Phi(X(s))+\frac{|v^{\prime}|^2}{2})^{\beta/2}
 \geq \left\{ \Phi(X(s))^2 + \Phi(X(s))\left(\frac{|u^{\prime}|^2}{2}+ \frac{|v^{\prime}|^2}{2}\right)
\right\}^{\beta /2} \ \ \ \ \ \ \ \ \ \ \\
=\left\{ \Phi(X(s))^2 + \Phi(X(s))\left(\frac{|u|^2}{2} + \frac{|V(s)|^2}{2}\right)\right\}^{\beta/2} \geq\left\{ \Phi(X(s)) \left(\Phi(X(s))+\frac{|V(s)|^2}{2}\right)
\right\}^{\beta /2}
\geq \Phi(X(s))^{\beta/2} w(X(s),V(s)). \ \ \ \ \ \  \ \ \
\end{eqnarray*}
Note that
\begin{eqnarray*}
e^{-\frac{1}{2}\int_s^t e^{-\Phi(X(\tau))}\nu(V(\tau))d\tau} &\leq& e^{-\frac{1}{2} \nu_0 e^{-|\Phi|_{\infty}}(t-s)},\\
2\frac{d}{ds} \left\{ e^{-\frac{1}{2}\int_s^t e^{-\Phi(X(\tau))} \nu(V(\tau)) d\tau }
\right\} &=& \nu(V(s)) e^{-\Phi(X(s))} e^{-\frac{1}{2}\int_s^t e^{-\Phi(X(\tau))} \nu(V(\tau)) d\tau}.
\end{eqnarray*}
Using above relations, we have an upper bound of the integrand of (\ref{3}) as
\begin{eqnarray*}
&&e^{-\frac{\Phi(X(s))}{2}} e^{-\frac{1}{2}\int_s^t e^{-\Phi(X(\tau))}\nu(V(\tau))d\tau} e^{-\frac{1}{2} \nu_0 e^{-|\Phi|_{\infty}}(t-s)} \nu(V(s)) ||h(s)||_{\infty}^2\\
&& \leq \ \ e^{\frac{|\Phi|_{\infty}}{2}} \nu(V(s)) e^{-\Phi(X(s))} e^{-\frac{1}{2}\int_s^t e^{-\Phi(X(\tau))}\nu(V(\tau))d\tau} \times \{ e^{-\frac{1}{4}\nu_0 e^{-|\Phi|_{\infty}}(t-s) } ||h(s)||_{\infty}
\}^2\\
&& = \ \ e^{\frac{|\Phi|_{\infty}}{2}} \times 2\frac{d}{ds} \left\{ e^{-\frac{1}{2}\int_s^t e^{-\Phi(X(\tau))} \nu(V(\tau)) d\tau }
\right\}\times \{ e^{-\frac{1}{4}\nu_0 e^{-|\Phi|_{\infty}}(t-s) } ||h(s)||_{\infty}
\}^2.
\end{eqnarray*}
Therefore we have a control of the integration of (\ref{3}) by
\begin{eqnarray}
(\ref{3})&\leq&
2 e^{\frac{|\Phi|_{\infty}}{2}}
\{ 1- e^{-\frac{1}{2}\int_0^t e^{-\Phi(X(\tau))} \nu(V(\tau))d\tau}
\}  \times \sup_{0\leq s\leq t}\{ e^{-\frac{1}{4}\nu_0 e^{-|\Phi|_{\infty}}(t-s) } ||h(s)||_{\infty}
\}^2
\nonumber\\
&\leq& 2 e^{\frac{|\Phi|_{\infty}}{2}} \times \sup_{0\leq s\leq t}\{ e^{-\frac{1}{4}\nu_0 e^{-|\Phi|_{\infty}}(t-s) } ||h(s)||_{\infty}
\}^2.
\label{*1}
\end{eqnarray}

From now, we concentrate to estimate (\ref{2}). Let $\mathbf{k}(v,v^{\prime})$ be the corresponding kernel associated with $K$.
Notice that in the integrand of (\ref{2})
\begin{equation}
\{K_w h\}(s,X(s),V(s))= \int_{\mathbb{R}^d} \mathbf{k}_w (V(s),v^{\prime})\underline{h(s,X(s),v^{\prime})} dv^{\prime}.\label{sharp1}
\end{equation}
Now we use the representation of the underlined $\underline{h(s,X(s),v^{\prime})}$ again to evaluate (\ref{sharp1}).
We need a following crucial inequality, Lemma 3 in \cite{Guo08} :
\begin{lemma}\label{crucial}(\cite{Guo08})
\begin{equation*}
|\mathbf{k}(v,v^{\prime})| \leq C \{|v-v^{\prime}|+ |v-v^{\prime}|^{-1}\}\exp\left\{ -\frac{1}{8}|v-v^{\prime}|^2 -\frac{1}{8}\frac{||v|^2 - |v^{\prime}|^2|^2}{|v-v^{\prime}|^2}
\right\}.
\end{equation*}
Let $0\leq \theta < \frac{1}{4}$. Then there exists $0\leq \varepsilon(\theta)<1$ and $C_{\theta}>0$ such that for $0\leq \varepsilon < \varepsilon(\theta)$,
\begin{equation*}
\int_{\mathbb{R}^d} \{|v-v^{\prime}|+ |v-v^{\prime}|^{-1} \} \exp\left\{ -\frac{1-\varepsilon}{8}|v-v^{\prime}|^2 - \frac{1-\varepsilon}{8}\frac{||v|^2 -|v^{\prime}|^2|^2}{|v-v^{\prime}|^2}
\right\}\frac{w(x,v)e^{\theta |v|^2}}{w(x,v^{\prime})e^{\theta |v^{\prime}|^2}}
dv^{\prime} \leq \frac{C}{1+|v|}.
\end{equation*}
\end{lemma}
\begin{proof}
We can check
\begin{eqnarray*}
\left|\frac{w(x,v)}{w(x,v^{\prime})}\right| &=& \left\{ \frac{\frac{|v|^2}{2} + \Phi(x)}{\frac{|v^{\prime}|^2}{2} + \Phi(x)} \right\}^{\beta /2}
\leq \left\{ 1+ \frac{\frac{|v^{\prime}|^2}{2} + \frac{|v-v^{\prime}|^2}{2}}{\Phi(x)+\frac{|v^{\prime}|^2}{2}}
\right\}^{\beta/2}
\leq \left\{1 + \frac{\frac{|v^{\prime}|^2}{2}}{\Phi(x)+ \frac{|v^{\prime}|^2}{2}} + \frac{\frac{|v-v^{\prime}|^2}{2}}{\Phi(x)+ \frac{|v^{\prime}|^2}{2}}\right\}^{\beta /2}\\
 &\leq& \{1+ 1 + C_{\Phi}|v-v^{\prime}|^2 \}^{\beta/2} \leq C(1+|v-v^{\prime}|^2)^{\beta/2}.
\end{eqnarray*}
The remainder of the proof is exactly same as the proof of Lemma 3 in \cite{Guo08}.
\end{proof}

In order to simplify notations, we define
\begin{eqnarray*}
\begin{cases}
&\Psi_1 (t) = \int_0^t e^{-\Phi(X(\tau;t,x,v))} \nu(V(\tau;t,x,v)) d\tau\\
&\Psi_2 (s,t) = \Phi(X(s;t,x,v))+  \int_s^t e^{-\Phi(X(\tau;t,x,v))} \nu(V(\tau;t,x,v)) d\tau \geq \nu_0 e^{-|\Phi|_{\infty}}(t-s)\\
&\Psi_3 (s,t)= \frac{1}{2}\Phi(X(s;t,x,v)) +  \int_s^t e^{-\Phi(X(\tau;t,x,v))} \nu(V(\tau;t,x,v)) d\tau\\
&\Psi_1^{\prime}(s)  = \int_0^s e^{-\Phi(X(\tau;s,X(s),v^{\prime}))} \nu(V(\tau;s,X(s),v^{\prime})) d\tau \geq \nu_0 e^{-|\Phi|_{\infty}} s\\
&\Psi_2^{\prime}(s_1 ,s)  =
\Phi(X(s_1;s,X(s),v^{\prime}))+ \int_{s_1}^s e^{-\Phi(X(\tau;s,X(s),v^{\prime}))}\nu(V(\tau;s,X(s),v^{\prime})) d\tau
\geq \nu_0 e^{-|\Phi|_{\infty}}(s-s_1)
\\
&\Psi_3^{\prime}(s_1 ,s) = \frac{1}{2}\Phi(X(s_1;s,X(s),v^{\prime}))+ \int_{s_1}^s e^{-\Phi(X(\tau;s,X(s),v^{\prime}))}\nu(V(\tau;s,X(s),v^{\prime})) d\tau
\end{cases}
\end{eqnarray*}
We can rewrite $\underline{h(s,X(s),v^{\prime})}$ in (\ref{sharp1}) as
\begin{eqnarray*}
e^{-\Psi_1^{\prime}(s)}h(0,X^{\prime}(0),V^{\prime}(0))
&+& \int_0^s e^{-\Psi_2^{\prime}(s_1,s)}\int_{\mathbb{R}^d} \mathbf{k}_w (V^{\prime}(s_1),v^{\prime\prime})h(s_1 ,X^{\prime}(s_1),v^{\prime\prime}) dv^{\prime\prime} ds_1\label{2''}\\
&+&\int_0^s e^{-\Psi_3^{\prime}(s_1,s)}w\Gamma\left(\frac{h}{w},\frac{h}{w}\right)(s_1,X^{\prime}(s_1),V^{\prime}(s_1))ds_1.\nonumber
\end{eqnarray*}
We plug the above formula into (\ref{sharp1}) and (\ref{2}) to have
\begin{eqnarray}
(\ref{2})&=&\int_0^t \int_{\mathbb{R}^d} e^{-\Psi_2 (t,s)} e^{-\Psi_1^{\prime}(s)}  \mathbf{k}_w (V(s),v^{\prime}) h_0 (X^{\prime}(0),V^{\prime}(0)) dv^{\prime} ds\label{1-1}\\
&+& \int_0^t \int_{\mathbb{R}^d} \int_0^s \int_{\mathbb{R}^d} e^{-\Psi_2 (t,s)} e^{-\Psi_2^{\prime}(s_1,s)} \mathbf{k}_w (V(s),v^{\prime})\mathbf{k}_w(V^{\prime}(s_1),v^{\prime\prime}) h(s_1,X^{\prime}(s_1),v^{\prime\prime})dv^{\prime\prime} ds_1 dv^{\prime} ds\label{1-2}\\
&+&\int_0^t \int_{\mathbb{R}^d} \int_0^s e^{-\Psi_2 (t,s)} e^{-\Psi_3^{\prime}(s_1,s)}\mathbf{k}_w(V(s),v^{\prime})w\Gamma\left(\frac{h}{w},\frac{h}{w}\right)(s_1,X^{\prime}(s_1),V^{\prime}(s_1))
ds_1 dv^{\prime} ds.\label{1-3}
\end{eqnarray}
For the first line, we have
\begin{eqnarray}
(\ref{1-1}) \leq \int_0^t e^{-t\nu_0 e^{-|\Phi|_{\infty}}} ||h_0||_{\infty} \int_{\mathbb{R}^d} \mathbf{k}_w (V(s),v^{\prime})dv^{\prime} ds
\leq \left( t e^{-\frac{t}{2}\nu_0 e^{-|\Phi|_{\infty}}}
\right)e^{-\frac{\nu_0 e^{-|\Phi|_{\infty}}}{2}t}
||h_0||_{\infty}\leq C e^{-\frac{\nu_0 e^{-|\Phi|_{\infty}}}{2}t}
||h_0||_{\infty},
\label{*2}
\end{eqnarray}
and the third line (\ref{1-3}) is bounded by
\begin{eqnarray}
\int_0^t  \int_{\mathbb{R}^d}  \int_0^s
e^{-\Phi(X(s))} \underbrace{e^{-\int_s^t e^{-\Phi(X(\tau))}\nu(V(\tau))d\tau}}_{\mathbf{I}}
\underbrace{e^{-\frac{1}{2}\Phi(X^{\prime}(s_1))} e^{-\int_{s_1}^s e^{-\Phi(X^{\prime}(\tau))}\nu(V^{\prime}(\tau))d\tau}\nu(V^{\prime}(s_1))}_{\mathbf{II}}\mathbf{k}_w (V(s),v^{\prime}) || h(s_1) ||_{\infty}^2 ds_1 dv^{\prime} ds\nonumber\\
\leq
{2}\ e^{\frac{|\Phi|_{\infty}}{2}} \int_0^t \ e^{-\frac{\nu_0 (t-s)}{2e^{|\Phi|_{\infty}}}} \int_{\mathbb{R}^d} \ \mathbf{k}_w (V(s),v^{\prime}) \int_0^s  \ \frac{d}{ds_1} \left\{e^{-\frac{1}{2}\int_{s_1}^s e^{-\Phi(X^{\prime}(\tau))}\nu(V^{\prime}(\tau))d\tau}\right\}ds_1dv^{\prime}ds
\sup_{0\leq s_1 \leq t} \{ e^{-\frac{\nu_0 (t-s_1)}{4e^{|\Phi|_{\infty}}}}||h(s_1)||_{\infty}
\}^2\nonumber\\
\leq
{2}e^{\frac{|\Phi|_{\infty}}{2}} \frac{2}{\nu_0}e^{|\Phi|_{\infty}} \times \sup_{0\leq s_1 \leq t} \{ e^{-\frac{\nu_0 e^{-|\Phi|_{\infty}}}{4}(t-s_1)}||h(s_1)||_{\infty}
\}^2 \ \ \leq \ \ \ \ \ \frac{4}{\nu_0} e^{\frac{3}{2}|\Phi|_{\infty}} \sup_{0\leq s_1 \leq t} \{ e^{-\frac{\nu_0 e^{-|\Phi|_{\infty}}}{4}(t-s_1)}||h(s_1)||_{\infty}
\}^2, \ \ \ \ \ \ \ \  \label{*3}
\end{eqnarray}
where we used
\begin{eqnarray*}
\mathbf{I} &\leq& e^{-{(t-s)\nu_0 e^{-|\Phi|_{\infty}}}},\\
\mathbf{II} &\leq&
2e^{\frac{|\Phi|_{\infty}}{2}}  e^{-\frac{\nu_0 e^{-|\Phi|_{\infty}}}{2}(s-s_1)}\times\frac{1}{2}e^{-\Phi(X^{\prime}(s_1))} e^{-\frac{1}{2}\int_{s_1}^s e^{-\Phi(X^{\prime}(\tau))}\nu(V^{\prime}(\tau))d\tau}   \nu(V^{\prime}(s_1))\\
&=&
2e^{\frac{|\Phi|_{\infty}}{2}} e^{-\frac{\nu_0 e^{-|\Phi|_{\infty}}}{2}(s-s_1)} \times \frac{d}{ds_1}\left\{e^{-\frac{1}{2}\int_{s_1}^s e^{-\Phi(X^{\prime}(\tau))}\nu(V^{\prime}(\tau))d\tau}\right\}.
\end{eqnarray*}
Now we concentrate on the second term, (\ref{1-2}).
\subsection{Estimate of (\ref{1-2})}
\textbf{CASE 1 : } $|v|\geq N$ with $N>> \sqrt{|\Phi|_{\infty}}$. Using (\ref{boundofV}), we have $|V(s)|\geq \frac{N}{2}$ so that
\begin{equation*}
\iint_{\mathbb{R}^d \times \mathbb{R}^d} \mathbf{k}_w (V(s),v^{\prime})\mathbf{k}_w (V^{\prime}(s_1),v^{\prime\prime}) dv^{\prime\prime} dv^{\prime} \leq \frac{C}{1+|V(s)|}\leq \frac{C}{N}.
\end{equation*}
Thus in this case, (\ref{1-2}) is bounded by
\begin{eqnarray}
\int_0^t  \int_0^s  e^{-\frac{\nu_0 e^{-|\Phi|_{\infty}}}{2}(t-s_1)} \iint_{\mathbb{R}^d \times \mathbb{R}^d } \mathbf{k}_w (V(s),v^{\prime}) \mathbf{k}_w (V^{\prime}(s_1),v^{\prime\prime}) dv^{\prime\prime}dv^{\prime}ds_1ds\times \sup_{0\leq s_1 \leq t} e^{-\frac{\nu_0 e^{-|\Phi|_{\infty}}}{2}(t-s_1)}||h(s_1)||_{\infty}\nonumber\\
\leq
\left(\frac{2}{\nu_0 e^{-|\Phi|_{\infty}}}\right)^2 \times \frac{C}{N} \times \sup_{0\leq s_1 \leq t} e^{-\frac{\nu_0 e^{-|\Phi|_{\infty}}}{2}(t-s_1)}||h(s_1)||_{\infty}\leq \frac{C}{N}\frac{4}{\nu_0^2}e^{2|\Phi|_{\infty}} \sup_{0\leq s_1 \leq t} e^{-\frac{\nu_0 e^{-|\Phi|_{\infty}}}{2}(t-s_1)}||h(s_1)||_{\infty},\label{*4}
\end{eqnarray}
where we used the fact
\begin{eqnarray*}
&&\int_0^t \int_0^s  e^{-\frac{\nu_0 e^{-|\Phi|_{\infty}}}{2}(t-s_1)}ds_1 ds\\
&& \ \ \ \ =\frac{2}{\nu_0 e^{-|\Phi|_{\infty}}} \int_0^t ds \int_0^{s}  \frac{d}{ds_1}\left\{ e^{-\frac{\nu_0 e^{-|\Phi|_{\infty}}}{2}(t-s_1)}\right\}ds_1= \frac{2}{\nu_0 e^{-|\Phi|_{\infty}}} e^{-\frac{\nu_0 e^{-|\Phi|_{\infty}}}{2}t} \int_0^t \{e^{\frac{\nu_0 e^{-|\Phi|_{\infty}}}{2}s}-1\}ds \\
&& \ \ \ \ =\frac{2}{\nu_0 e^{-|\Phi|_{\infty}}} e^{-\frac{\nu_0 e^{-|\Phi|_{\infty}}}{2}t}\left(
\frac{2}{\nu_0 e^{-|\Phi|_{\infty}}} e^{\frac{\nu_0 e^{-|\Phi|_{\infty}}}{2}t} -t - \frac{2}{\nu_0 e^{-|\Phi|_{\infty}}}
\right)\leq \left(\frac{2}{\nu_0 e^{-|\Phi|_{\infty}}}\right)^2.
\end{eqnarray*}
\textbf{CASE 2 : } $|v|\leq N, |v^{\prime}|\geq 2N,$ or $|v^{\prime}|\leq 2N, |v^{\prime\prime}|\geq 3N$. Observe that
\begin{eqnarray*}
|V(s)-v^{\prime}| &\geq& |v^{\prime}-v|-|V(s)-v|\geq |v^{\prime}|-|v|-|V(s)-v|,\\
|V^{\prime}(s_1)-v^{\prime\prime}|&\geq& |v^{\prime\prime}-v^{\prime}|-|V^{\prime}(s_1)-v^{\prime}|\geq |v^{\prime\prime}|-|v^{\prime}|-|V^{\prime}(s_1)-v^{\prime}|,
\end{eqnarray*}
 and $|V(s)-v|, |V^{\prime}(s_1)-v^{\prime}|\leq 2 |\Phi|_{\infty}^2$ from (\ref{boundofV}), thus we have either $|V(s)-v^{\prime}|\geq \frac{N}{2}$ or $|V^{\prime}(s_1)-v^{\prime\prime}|\geq \frac{N}{2}$ and either one of the followings are valid correspondingly for $\eta>0$:
\begin{eqnarray*}
\mathbf{k}_w (V(s),v^{\prime}) &\leq& e^{-\frac{\eta}{8}N^2} \mathbf{k}_w (V(s),v^{\prime}) e^{\frac{\eta}{8}|V(s)-v^{\prime}|^2},\\
\mathbf{k}_w (V^{\prime}(s_1),v^{\prime\prime}) &\leq& e^{-\frac{\eta}{8}N^2} \mathbf{k}_w (V^{\prime}(s_1),v^{\prime\prime}) e^{\frac{\eta}{8}|V^{\prime}(s_1)-v^{\prime\prime}|^2}.
\end{eqnarray*}
From Lemma \ref{crucial}, both $ \int  \mathbf{k}_w (V(s),v^{\prime}) e^{\frac{\eta}{8}|V(s)-v^{\prime}|^2}dv^{\prime}$ and $\int \mathbf{k}_w (V^{\prime}(s_1),v^{\prime\prime}) e^{\frac{\eta}{8}|V^{\prime}(s_1)-v^{\prime\prime}|^2} dv^{\prime\prime}$ are still finite for sufficiently small $\eta>0$. Therefore $(\ref{1-2})$ is bounded by
\begin{eqnarray}
 \int_0^t ds \int_0^s ds_1 e^{-\frac{\nu_0 e^{-|\Phi|_{\infty}}}{2}(t-s_1)} \iint dv^{\prime\prime}dv^{\prime} \mathbf{k}_w (V(s),v^{\prime}) \mathbf{k}_w (V^{\prime}(s_1),v^{\prime\prime}) \sup_{0\leq s_1 \leq t} e^{-\frac{\nu_0 e^{-|\Phi|_{\infty}}}{2}(t-s_1)}||h(s_1)||_{\infty} \ \ \ \ \ \ \ \ \ \ \ \nonumber\\
\leq
 \left(\frac{2}{\nu_0 e^{-|\Phi|_{\infty}}}\right)^2  e^{-\frac{\eta}{8}N^2} \sup_{0\leq s_1 \leq t} e^{-\frac{\nu_0 e^{-|\Phi|_{\infty}}}{2}(t-s_1)}||h(s_1)||_{\infty}
\leq
e^{-\frac{\eta}{8}N^2}\frac{4}{\nu_0^2}e^{2|\Phi|_{\infty}} \sup_{0\leq s_1 \leq t} e^{-\frac{\nu_0 e^{-|\Phi|_{\infty}}}{2}(t-s_1)}||h(s_1)||_{\infty},\label{*5}
\end{eqnarray}
where we used the fact
\begin{eqnarray*}
\iint  \mathbf{k}_w (V(s),v^{\prime}) \mathbf{k}_w (V^{\prime}(s_1),v^{\prime\prime})dv^{\prime\prime}dv^{\prime} \leq e^{-\frac{\eta}{8}N^2}\iint \mathbf{k}_w (V(s),v^{\prime}) \mathbf{k}_w (V^{\prime}(s_1),v^{\prime\prime})\{e^{\frac{\eta}{8}|V(s)-v^{\prime}|^2}+ e^{\frac{\eta}{8}|V^{\prime}(s_1)-v^{\prime\prime}|^2}\} \leq C e^{-\frac{\eta}{8}N^2}.
\end{eqnarray*}
\textbf{CASE 3 : } $|v|\leq N, |v^{\prime}|\leq 2N , |v^{\prime\prime}|\leq 3N$. This is the last remaining case because if $|v^{\prime}|> 2N$, it is included in Case 2; while if $|v^{\prime\prime}|> 3N$, either $|v^{\prime}|\leq 2N$ or $|v^{\prime}|\geq 2N$ are also included in Case 2.  We can bound $(\ref{1-2})$ by
\begin{eqnarray}
 \int_0^t
\int_{|v^{\prime}|\leq 2N}
\int_0^s
e^{-\nu_0 e^{-|\Phi|_{\infty}}(t-s_1)} \int_{|v^{\prime\prime}|\leq 3N}  \underbrace{\mathbf{k}_w (V(s),v^{\prime}) \mathbf{k}_w (V^{\prime}(s_1),v^{\prime\prime})}_{\bigodot}
 |h(s_1,X^{\prime}(s_1),v^{\prime\prime})| dv^{\prime\prime} ds_1 dv^{\prime} ds. \label{ooo}
\end{eqnarray}
Since $\mathbf{k}_w(v,v^{\prime})$ has possible integrable singularity of $\frac{1}{|v-v^{\prime}|}$, we can choose a smooth function with compact support $\mathbf{k}_N (v,v^{\prime})$ such that
\begin{equation*}
\sup_{|p|\leq 3N} \int_{|v^{\prime}|\leq 3N} |\mathbf{k}_w(p,v^{\prime})-\mathbf{k}_N (p,v^{\prime})| dv^{\prime} \leq \frac{1}{N}.
\end{equation*}
Splitting $\mathbf{k}_w (V(s),v^{\prime}) \mathbf{k}_w (V^{\prime}(s_1),v^{\prime\prime})$ in $\bigodot$ by
\begin{eqnarray}
\mathbf{k}_N(V(s),v^{\prime}) \mathbf{k}_N (V^{\prime}(s_1),v^{\prime\prime}) \ , \ \ \ \ \ \ \ \ \ \ \ \ \ \ \ \ \ \ \ \ \ \ \ \ \ \ \ \ \ \ \ \ \ \ \ \label{o1}\\
\{\mathbf{k}_w (V(s),v^{\prime})- \mathbf{k}_N (V(s),v^{\prime})\} \mathbf{k}_w (V^{\prime}(s_1),v^{\prime\prime})
+ \{\mathbf{k}_{w}(V^{\prime}(s_1),v^{\prime\prime})-\mathbf{k}_N (V^{\prime}(s_1),v^{\prime\prime})\}\mathbf{k}_N (V(s),v^{\prime}).\label{o2}
\end{eqnarray}
We can bound (\ref{ooo}), in the case of $\bigodot = (\ref{o2})$, by
\begin{equation}
\frac{C}{N}\frac{4}{\nu_0^2}e^{2|\Phi|_{\infty}} \sup_{0\leq s_1 \leq t} e^{-\frac{\nu_0 e^{-|\Phi|_{\infty}}}{2}(t-s_1)} || h(s_1) ||_{\infty}.\label{**1}
\end{equation}
In the case of $\bigodot = (\ref{o1})$, we can bound (\ref{ooo}) by
\begin{eqnarray}
{C_N}\int_0^t ds
\int_{|v^{\prime}|\leq 2N} dv^{\prime}
\int_0^s ds_1
e^{-{\nu_0 e^{-|\Phi|_{\infty}}}(t-s_1)}
 \int_{|v^{\prime\prime}|\leq 3N}
|h(s_1 ,X(s_1 ; s, X(s;t,x,v),v^{\prime}),v^{\prime\prime})|dv^{\prime\prime}.\label{main}
\end{eqnarray}
Recall that we need to show the decay for $t=T_0$ from (\ref{expdecayT}). Since the potential is time-independent we have
\begin{equation*}
X(s_1 ; s , X(s;T_0,x,v) , v^{\prime}) = X(s_1-s+T_0 ; T_0 , X(T_0 ;2T_0 -s ,x,v),v^{\prime}),
\end{equation*}
for $0\leq s_1 \leq s \leq T_0$. From Lemma \ref{epsilonneigh}, we split $(\ref{main})$ by
\begin{eqnarray}
{C_N} \sum_{i^1}^{M^1} \sum_{I^2}^{(M^2)^d}\sum_{I^3}^{(M^3)^d} \int_0^{T_0} ds \  \int_0^{s} ds_1 \mathbf{1}_{\{X(s_1-s+T_0 ;T_0,x,v)\in \mathfrak{D}_{I^2}^2\}}(s_1,s)\ \underbrace{\mathbf{1}_{\mathfrak{D}_{i^1}^1}(s_1-s+T_0)}_{\bigotimes} e^{-{\nu_0 e^{-|\Phi|_{\infty}}}(T_0-s_1)} \ \nonumber\\
\times\int_{|v^{\prime}|\leq 2N} dv^{\prime} \mathbf{1}_{\mathfrak{D}_{I^3 }^3}(v^{\prime}) \int_{|v^{\prime\prime}|\leq 3N} |h(s_1,
X(s_1-s+T_0 ; T_0 , X(T_0 ;2T_0 -s ,x,v),v^{\prime})
,v^{\prime\prime})|
dv^{\prime\prime}.\label{oo}
\end{eqnarray}
From Lemma \ref{epsilonneigh}, we have
\begin{eqnarray*}
\{(s_1-s+T_0, X(T_0 ;2T_0 -s ,x,v),v^{\prime})\in \mathfrak{D}_{i^1}^1\times \mathfrak{D}^2_{I^2 } \times \mathfrak{D}^3_{I^3 }  : det\left(\frac{\partial X}{\partial v^{\prime}}\right)(s_1-s+T_0 ; T_0 , X(T_0 ;2T_0 -s ,x,v),v^{\prime})=0
\}\\
\subset
\bigcup_j^d
\{(s_1-s+T_0, X(T_0 ;2T_0 -s ,x,v),v^{\prime})\in \mathfrak{D}_{i^1}^1\times \mathfrak{D}^2_{I^2 } \times \mathfrak{D}^3_{I^3 } \ : \ s_1-s +T_0 \in (t_{j ,i^1, I^2 , I^3}-\frac{\varepsilon}{4 M^1}, t_{j ,i^1, I^2, I^3 }+\frac{\varepsilon}{4 M^1})
\}.
\end{eqnarray*}
For each $i^1, I^2 ,I^3 $ and $j$, we split $\bigotimes$ as
\begin{eqnarray}
\mathbf{1}_{\mathfrak{D}_{i^1}^1}(s_1-s+T_0)\mathbf{1}_{(t_{j,i^1, I^2, I^3}- \frac{\varepsilon}{4 M^1}, t_{j,i^1, I^2, I^3}+ \frac{\varepsilon}{4 M^1} )}( s_1-s+T_0)  \ , \ \ \ \ \ \label{1-a}\\
 \mathbf{1}_{\mathfrak{D}_{i^1}^1}(s_1-s + T_0)\{1- \mathbf{1}_{(t_{j,i^1, I^2, I^3}- \frac{\varepsilon}{4 M^1}, t_{j,i^1, I^2, I^3}+ \frac{\varepsilon}{4 M^1} )}( s_1-s+T_0)\}.\label{1-b}
\end{eqnarray}
\textbf{CASE 3a : } In the case of $\bigotimes =(\ref{1-a})$, the integration (\ref{oo}) is bounded by
\begin{eqnarray*}
{C_N}\sum_{i^1 }^{M^1} \sum_{I^2}^{(M^2)^d} \sum_{I^3}^{(M^3)^d} \int_0^{T_0} ds
\int_0^{s} ds_1 \mathbf{1}_{\{X(s_1-s+T_0;T_0,x,v)\in \mathfrak{D}_{I^2}^2\}}(s_1,s)\ \mathbf{1}_{\mathfrak{D}_{i^1}^1}(s_1 -s+T_0) \underline{e^{-\nu_0 e^{-|\Phi|_{\infty}}(T_0-s_1)}}_{*}\ \ \ \ \ \ \ \ \ \ \ \ \ \  \\ \times \mathbf{1}_{(t_{j,i^1, I^2, I^3}- \frac{\varepsilon}{4 M^1}, t_{j,i^1, I^2, I^3}+ \frac{\varepsilon}{4 M^1} )}( s_1-s+T_0) \
 \ \ \ \ \ \ \ \ \ \ \ \ \ \ \ \ \ \ \ \ \ \ \ \ \ \ \ \ \ \ \ \ \ \ \ \ \ \ \ \ \ \ \
\\
\times\int_{|v^{\prime}|\leq 2N} dv^{\prime} \mathbf{1}_{\mathfrak{D}_{I^3 }^3}(v^{\prime}) \int_{|v^{\prime\prime}|\leq 3N} |h(s_1,
X(s_1-s+T_0 ; T_0 , X(T_0 ;2T_0 -s ,x,v),v^{\prime})
,v^{\prime\prime})|
dv^{\prime\prime}.
\end{eqnarray*}
We split
\begin{eqnarray*}
\underline{e^{-{\nu_0 e^{-|\Phi|_{\infty}}}(t-s_1)}}_{*} = e^{-{\nu_0 e^{-|\Phi|_{\infty}}}{2}(t-s)} e^{-\frac{\nu_0 e^{-|\Phi|_{\infty}}}{2}(s-s_1)} \times e^{-\frac{\nu_0 e^{-|\Phi|_{\infty}}}{2}(t-s_1)}.
\end{eqnarray*}
and rewrite the above integration as
\begin{eqnarray*}
{C_N}\sum_{i^1 }^{M^1} \sum_{I^2}^{(M^2)^d} \sum_{I^3}^{(M^3)^d} \int_0^{T_0} ds \mathbf{1}_{\{X(s_1-s+T_0;T_0,x,v)\in \mathfrak{D}_{I^2}^2\}}(s_1,s)e^{-\frac{\nu_0 e^{-|\Phi|_{\infty}}}{2}(T_0-s)} \ \ \ \ \ \ \ \ \ \ \ \ \ \  \ \ \ \ \ \ \ \ \ \ \ \ \ \ \ \ \ \ \ \ \ \ \ \ \ \ \ \ \ \ \ \ \ \ \ \ \ \ \ \ \ \ \ \ \ \ \ \ \ \ \ \ \ \ \ \ \ \ \ \ \ \ \ \ \ \\
\ \ \ \ \ \ \ \ \ \ \ \ \ \ \ \ \ \ \ \ \ \ \ \ \ \ \times \int_0^{s} ds_1 e^{-\frac{\nu_0 e^{-|\Phi|_{\infty}}}{2}(s-s_1)} \ \mathbf{1}_{\mathfrak{D}_{i^1}^1}(s_1 -s+T_0)
\mathbf{1}_{(t_{j,i^1, I^2, I^3}- \frac{\varepsilon}{4 M^1}, t_{j,i^1, I^2, I^3}+ \frac{\varepsilon}{4 M^1} )}( s_1-s+T_0) \
 \ \ \ \ \ \ \ \ \ \ \ \ \ \ \ \ \ \ \ \ \ \ \ \ \ \ \ \ \ \ \ \ \ \ \ \ \ \ \ \ \ \ \ \label{ES2}
\\
\times\int_{|v^{\prime}|\leq 2N} dv^{\prime} \mathbf{1}_{\mathfrak{D}_{I^3 }^3}(v^{\prime}) \int_{|v^{\prime\prime}|\leq 3N} e^{-\frac{\nu_0 e^{-|\Phi|_{\infty}}}{2}(T_0-s_1)} ||h(s_1)||_{\infty}
dv^{\prime\prime}. \ \ \ \ \ \ \ \ \ \ \ \ \ \ \ \ \ \ \ \ \ \ \ \ \ \ \ \ \ \ \ \ \ \ \ \ \ \ \ \ \ \ \ \ \ \ \ \ \ \ \ \
\ \ \ \ \ \ \ \ \ \ \ \ \ \ \ \ \ \ \ \ \ \ \ \ \
   \label{ES3}
\end{eqnarray*}
For fixed $i^1 , I^2 , I^3$, using the fact that $e^{-\frac{\nu_0 e^{-|\Phi|_{\infty}}}{2}(s-s_1)}$ is an increasing function of $s_1 \in [0,s]$, the second line of the above term is bounded by
\begin{eqnarray*}
 \int_{s-\frac{ \varepsilon}{4M^1}}^s \frac{2}{\nu_0 e^{-|\Phi|_{\infty}}} \frac{d}{ds_1}\left\{e^{-\frac{\nu_0 e^{-|\Phi|_{\infty}}}{2}(s-s_1)}\right\} ds_1
= \frac{2}{\nu_0 e^{-|\Phi|_{\infty}}} \{1-e^{-\frac{\nu_0 e^{-|\Phi|_{\infty}}}{2}\frac{\varepsilon}{4M^1}}\} \sim \frac{2}{\nu_0 e^{-|\Phi|_{\infty}}} \frac{\nu_0 e^{-|\Phi|_{\infty}}}{2}\frac{ \varepsilon}{4M^1} =\frac{\varepsilon}{4M^1}.
\end{eqnarray*}
Therefore we can bound $(\ref{oo})$ by
\begin{eqnarray}
&&C_N \int_0^{T_0}  \underline{\sum_{I^2 } \mathbf{1}_{\{X(s_1-s+T_0;T_0,x,v)\in \mathfrak{D}_{i^2}^2\}}(s_1,s)}_{**} e^{-\frac{\nu_0 e^{-|\Phi|_{\infty}}}{2}(T_0-s)}ds \times \sum_{i^1 =1}^{M^1}\frac{\varepsilon}{4M^1}\nonumber\\
&& \ \ \ \ \ \ \ \times \int_{|v^{\prime}|\leq 2N} dv^{\prime} \underline{\sum_{I^3} \mathbf{1}_{\mathfrak{D}_{I^3 }^3}(v^{\prime})}_{***} \times(3N)^3 \sup_{0\leq s_1 \leq T_0}e^{-\frac{\nu_0 e^{-|\Phi|_{\infty}}}{2}(T_0-s_1)}|| h(s_1) ||_{\infty}\nonumber\\
&& \ \ \ \ \leq \ C_N \int_0^{T_0}  e^{-\frac{\nu_0 e^{-|\Phi|_{\infty}}}{2}(T_0-s)}ds \times \frac{\varepsilon}{4}
\int_{|v^{\prime}|\leq 2N} dv^{\prime}  (3N)^3 \sup_{0\leq s_1 \leq T_0}e^{-\frac{\nu_0 e^{-|\Phi|_{\infty}}}{2}(T_0-s_1)}|| h(s_1) ||_{\infty}\nonumber\\
&& \ \ \ \ \leq \ C_N \frac{2}{\nu_0 e^{-|\Phi|_{\infty}}}\frac{\varepsilon}{4} (2N)^3 (3N)^3 \sup_{0\leq s_1 \leq T_0} e^{-\frac{\nu_0 e^{-|\Phi|_{\infty}}}{2}(T_0-s_1)}||h(s_1)||_{\infty}\nonumber\\
&& \ \ \ \ \leq \
\varepsilon \frac{C_N e^{|\Phi|_{\infty}}}{\nu_0} \sup_{0\leq s_1 \leq T_0} e^{-\frac{\nu_0 e^{-|\Phi|_{\infty}}}{2}(T_0-s_1)}||h(s_1)||_{\infty},\label{*6}
\end{eqnarray}
where we used the fact that
\begin{eqnarray*}
&\underline{**}& \ \ \ \ \  \sum_{I^2}^{(M^2)^d} \mathbf{1}_{\{X(s_1-s+T_0;T_0 ,x,v)\in \mathfrak{D}_{I^2 }^2\}}(s_1,s) = \mathbf{1}_{\{X(s_1-s+T_0;T_0 ,x,v) \in \mathbb{T}^d\}}(s_1,s) =\mathbf{1}_{\{0\leq s \leq T_0\}}(s) \mathbf{1}_{\{0\leq s_1 \leq s\}}(s_1),\\
&\underline{***}& \ \ \ \ \ \sum_{I^3}^{(M^3)^d} \mathbf{1}_{\mathfrak{D}_{I^3}^3}(v^{\prime}) \mathbf{1}_{|v^{\prime}|\leq 2N}(v^{\prime}) = \mathbf{1}_{|v^{\prime}|\leq 2N} (v^{\prime}).
\end{eqnarray*}
\textbf{CASE 3b : } In the case of $\bigotimes= (\ref{1-b})$, the integration (\ref{oo}) is bounded by
\begin{eqnarray}
{C_N}\sum_{i^1 }^{M^1} \sum_{I^2}^{(M^2)^d} \sum_{I^3}^{(M^3)^d} \int_0^{T_0} ds
\int_0^{s} ds_1 \mathbf{1}_{\{X(s_1-s+T_0;T_0,x,v)\in \mathfrak{D}_{I^2}^2\}}(s_1,s)\ \mathbf{1}_{\mathfrak{D}_{i^1}^1}(s_1 -s+T_0) e^{-\nu_0 e^{-|\Phi|_{\infty}}(T_0-s_1)}\ \ \ \ \ \ \ \ \ \ \ \ \ \  \nonumber\\
 \times \bigg\{1-\mathbf{1}_{(t_{j,i^1, I^2, I^3}- \frac{\varepsilon}{4 M^1}, t_{j,i^1, I^2, I^3}+ \frac{\varepsilon}{4 M^1} )}( s_1-s+T_0) \bigg\} \
 \ \ \ \ \ \ \ \ \ \ \ \ \ \ \ \ \ \ \ \ \ \ \ \ \ \ \ \ \ \ \ \ \ \ \ \ \ \label{crucial line}
\\
\times\int_{|v^{\prime}|\leq 2N} \mathbf{1}_{\mathfrak{D}_{I^3 }^3}(v^{\prime}) \int_{|v^{\prime\prime}|\leq 3N} |h(s_1,
\underline{X(s_1-s+T_0 ; T_0 , X(T_0 ;2T_0 -s ,x,v),v^{\prime})}
,v^{\prime\prime})|dv^{\prime}
dv^{\prime\prime}.\nonumber
\end{eqnarray}
By Lemma \ref{epsilonneigh}, we can apply a change of variables :
\begin{equation}
v^{\prime} \rightarrow y\equiv X(s_1-s+T_0 ; T_0 , X(T_0 ;2T_0 -s ,x,v),v^{\prime}),\nonumber
\end{equation}
\begin{equation*}
Jac \left(\frac{\partial X}{\partial v^{\prime}}\right) (s_1-s+T_0 ; T_0 , X(T_0 ;2T_0 -s ,x,v),v^{\prime}) > \delta_*.
\end{equation*}
Therefore the last line of the above term is bounded by
\begin{eqnarray*}
\frac{1}{\delta_*} \int_{x\in \mathbb{T}^d} \int_{|v^{\prime\prime}|\leq 3N}  w(x,v^{\prime\prime}) |f(s_1 ,x,v^{\prime\prime})| dv^{\prime\prime} dx  &\leq&
\frac{1}{\delta_*} \left( \int_{|v^{\prime\prime}|\leq 3N}\int_{\mathbb{T}^d} w(x,v^{\prime\prime})^2 dx dv^{\prime\prime}
\right)^{\frac{1}{2}} ||f(s_1)||_{L^2}\\
&\leq& \ \frac{1}{\delta_*} C(N,|\Phi|_{\infty}) ||f(s_1)||_{L^2}.
\end{eqnarray*}
Therefore, in the case of $\bigotimes=(\ref{1-b})$, we have an upper bound of (\ref{oo}) as
\begin{eqnarray}
C(M^1,M^2,M^3,\delta_*, N , |\Phi|_{\infty}, \nu_0) \int_0^{T_0} ||f(s_1)||_{L^2} ds_1.\label{*7}
\end{eqnarray}
\newline To summarize, let $\lambda = \frac{\nu_0 e^{-|\Phi|_{\infty}}}{4}$ and from (\ref{one}), (\ref{*1}), (\ref{*2}), (\ref{*3}), (\ref{*4}), (\ref{*5}), (\ref{**1}), (\ref{*6}) and (\ref{*7}) we conclude
\begin{eqnarray*}
||h(T_0)||_{\infty} &\leq& e^{-\lambda T_0} ||h_0||_{\infty} + C(\frac{1}{N}+\varepsilon + e^{-\frac{\eta}{8}N^2})\sup_{0\leq s\leq T_0} \{e^{-\lambda(T_0-s)}||h(s)||_{\infty}\}\\
&&+ \ C \sup_{0\leq s\leq T_0} \{e^{-\lambda(T_0-s)}||h(s)||_{\infty}\}^2
+ C_{T_0} \int_0^{T_0} ||f(s_1)||_{L^2} ds_1.
\end{eqnarray*}
Assume $\sup_{0\leq s\leq T_0} \{e^{\lambda s} ||h(s)||_{\infty}\}$ is sufficiently small. Choose sufficiently large $N>0$ and small $\varepsilon>0$ and small $||h_0||_{\infty}$. Then we conclude (\ref{expdecayT}).
\section{Nonlinear $L^{\infty}$ Stability}
In this section, we prove Theorem 2 and establish the nonlinear $L^{\infty}$ stability in (\ref{stability}). The following lemma, which has been established in \cite{Guo_short}, plays a crucial rule in the proof of the nonlinear stability (\ref{stability}) without the conservation of momentum.
\begin{lemma}(\cite{Guo_short})
Let $\mu_E(x,v)=\exp\{-\frac{|v|^2}{2}-\Phi(x)\}$. Assume $F$ satisfies the conservation of mass (\ref{masscons}), energy (\ref{energycons}) and the entropy inequality (\ref{entropy}). For $0<\delta<1$, we have
\begin{equation}
\iint |F(t)-\mu_E| \mathbf{1}_{|F(t)-\mu_E|\geq \delta \mu_E}
\leq \frac{4}{\delta}\{\mathcal{H}(F_0)-\mathcal{H}(\mu_E) + |M_0| + |E_0|\}.
\label{entropyineq}
\end{equation}
\end{lemma}
\begin{proof}
The proof is almost same as the argument in Page 147 of \cite{Guo_short}. The difference is the fact that $\ln \mu_E = -\frac{|v|^2}{2} -\Phi(x)$ is only bounded by the energy $\frac{|v|^2}{2}+\Phi(x)$. We now make use of the entropy inequality (\ref{entropy}). Recall from the Taylor expansion,
\begin{eqnarray*}
\mathcal{H}(F(t))-\mathcal{H}(\mu_E)
&=& \iint \{F(t)\ln F(t) - \mu_E \ln \mu_E\}
= \underline{\iint (\ln \mu_E +1) \{F(t)-\mu_E\} } +\iint\frac{\{F(t)-\mu_E\}^2}{2\tilde{F}}\\
&\leq&  \mathcal{H}(F_0)-\mathcal{H}(\mu_E) \ ,
\end{eqnarray*}
where $\tilde{F}$ is a number between $F(t)$ and $\mu_E$. Notice that the underlined term is bounded by the mass and energy of $F(t)$. Hence, from the conservation of mass (\ref{masscons}) and energy (\ref{energycons}), we get
\begin{equation*}
\iint \frac{\{F(t)-\mu_E\}^2}{2\tilde{F}} \leq \
\mathcal{H}(F_0) -\mathcal{H}(\mu_E) \ + \ |M_0| \ + \ |E_0|.
\end{equation*}
The rest of the proof is exactly same as the argument of Page 147 of \cite{Guo_short}.
\end{proof}

In order to obtain (\ref{stability}), we do estimate a weighted perturbation $h$ in (\ref{h}) satisfying the linearized Boltzmann equation (\ref{linearBE-h}). The proof is exactly same as Section 4 except \textbf{CASE 3b}. Consider (\ref{crucial line}) in \textbf{CASE 3b}. We introduce the indicator functions $\mathbf{1}_{|F(t)-\mu_E|\leq \delta \mu_E}$ and $\mathbf{1}_{|F(t)-\mu_E|\geq \delta \mu_E}$ and split the last line of (\ref{crucial line}) into
\begin{eqnarray}
\int_{|v^{\prime}|\leq 2N}\int_{|v^{\prime\prime}|\leq 3N} \mathbf{1}_{\mathfrak{D}_{I^3}^3} |h(s_1,X(s_1-s+T_0),v^{\prime\prime})| \mathbf{1}_{|F(t)-\mu_E|\leq \delta \mu_E}
+ \mathbf{1}_{\mathfrak{D}_{I^3}^3} |h(s_1,X(s_1-s+T_0),v^{\prime\prime})| \mathbf{1}_{|F(t)-\mu_E|\geq \delta \mu_E}.
\end{eqnarray}
The first integration is bounded by
\begin{eqnarray}
\delta \int_{|v^{\prime}|\leq 2N}\int_{|v^{\prime\prime}|\leq 3N} \mathbf{1}_{\mathfrak{D}_{I^3}^3} w(X(s_1-s+T_0),v^{\prime\prime}) \sqrt{\mu_E (X(s_1-s+T_0),v^{\prime\prime})}.\nonumber
\end{eqnarray}
Using Lemma 4, the second integration is bounded by
\begin{equation}
\int_{|v^{\prime}|\leq 2N} \int_{|v^{\prime\prime}|\leq 3N}
\mathbf{1}_{\mathfrak{D}_{I^3}^3} \frac{w}{\sqrt{\mu_E}}(X(s_1-s+T_0),v^{\prime\prime})
|F(s_1,X(s_1-s+T_0),v^{\prime\prime})-\mu_E(X(s_1-s+T_0),v^{\prime\prime})| \mathbf{1}_{|F(t)-\mu_E|\geq \delta \mu_E}
dv^{\prime\prime} dv^{\prime}.\label{aaa}
\end{equation}
By Lemma 2, we apply the change of variables
\begin{equation}
v^{\prime} \rightarrow y=X(s_1-s+T_0 ; T_0 , X(T_0 ;2T_0 -s ,x,v),v^{\prime}),\nonumber
\end{equation}
\begin{equation*}
Jac \left(\frac{\partial X}{\partial v^{\prime}}\right) (s_1-s+T_0 ; T_0 , X(T_0 ;2T_0 -s ,x,v),v^{\prime}) > \delta_*.
\end{equation*}
to bound (\ref{aaa}) by
\begin{eqnarray}
\frac{C_{N,\Phi}}{\delta_*} \int_{|v^{\prime}|\leq 2N} \int_{y\in\mathbb{T}^d} \mathbf{1}_{\mathfrak{D}_{I^3}^3} |F(s_1,y,v^{\prime\prime})-\mu_E(s_1,y,v^{\prime\prime})| \mathbf{1}_{|F(t)-\mu_E|\geq \delta\mu_E} dy dv^{\prime}.
\end{eqnarray}
Combining these two cases, using Lemma 4, the whole integration (\ref{crucial line}) is bounded by
\begin{eqnarray*}
&&C_{N,\Phi} \int_0^{T_0} ds \int_0^s ds_1 e^{-\nu_0 e^{-|\Phi|_{\infty}}(T_0-s_1)} \int_{|v^{\prime}|\leq 2N} \int_{\mathbb{T}^d} \{ \delta + \frac{1}{\delta_*} |F(s_1,y,v^{\prime\prime})-\mu_E(s_1,y,v^{\prime\prime})| \mathbf{1}_{|F(t)-\mu_E|\geq \delta\mu_E}
\}
dy dv^{\prime}\\
&& \ \ \ \leq \ C_{N,\Phi} \left[ \delta + \frac{1}{\delta_* \delta} {\{\mathcal{H}(F_0)-\mathcal{H}(\mu_E)+|M_0|+|E_0|\}}
\right] \ \leq \ C_{N,\Phi} \delta_*^{-\frac{1}{2}} \sqrt{\mathcal{H}(F_0)-\mathcal{H}(\mu_E)+|M_0|+|E_0|} \ .
\end{eqnarray*}
We also have optimized $\delta$ such that (for sufficiently small $|\mathcal{H}(F_0)-\mathcal{H}(\mu_E)|+|M_0|+|E_0|$),
\begin{equation*}
\delta= \frac{1}{\delta_* \delta} {\{\mathcal{H}(F_0)-\mathcal{H}(\mu_E)+|M_0|+|E_0|\}}.
\end{equation*}
To summarize, from the last part of Section 4, we conclude
\begin{eqnarray*}
\sup_{0\leq t\leq T_0}||h(t)||_{\infty} &\leq& e^{-\lambda T_0} ||h_0||_{\infty} + C(\frac{1}{N} + \varepsilon + e^{-\frac{\eta}{8}N^2}) \sup_{0\leq s\leq T_0} \{ e^{-\lambda(T_0 -s)}||h(s)||_{\infty}\}\\
&+& C \sup_{0\leq s\leq T_0} \{ e^{-\lambda(T_0 -s)}||h(s)||_{\infty}\}^2 + C_{N,\Phi} \delta_*^{-\frac{1}{2}} \sqrt{\mathcal{H}(F_0)-\mathcal{H}(\mu_E)+|M_0|+|E_0|}.
\end{eqnarray*}
Assume $\sup_{0\leq s\leq T_0} ||h(s)||_{\infty}$ and $\varepsilon>0$ sufficiently small and $T_0, N, \eta$ sufficiently large to conclude
\begin{equation}
|| h(T_0) ||_{\infty} \leq \frac{1}{2}||h_0||_{\infty} + C_{T_0} \sqrt{\mathcal{H}(F_0)-\mathcal{H}(\mu_E)+|M_0|+|E_0|}.\label{finite}
\end{equation}
From this finite time estimate, we use the argument in page 23 of \cite{E-G-M} to establish a large time estimate. Apply (\ref{finite}) repeatedly to get
\begin{eqnarray*}
||h(nT_0)||_{\infty} &\leq& \frac{1}{2}||h_0||_{\infty} + C_{T_0} \sqrt{\mathcal{H}(F_0)-\mathcal{H}(\mu_E)+|M_0|+|E_0|}\\
&\leq & \frac{1}{4}||h_0||_{\infty} + \{1+\frac{1}{2}\}C_{T_0} \sqrt{\mathcal{H}(F_0)-\mathcal{H}(\mu_E)+|M_0|+|E_0|}\\
&\leq& ...\\
&\leq& \frac{1}{2^n}||h_0||_{\infty} + \{1+\frac{1}{2}+ \frac{1}{4}+...\}C_{T_0} \sqrt{\mathcal{H}(F_0)-\mathcal{H}(\mu_E)+|M_0|+|E_0|}\\
&\leq& \frac{1}{2^n}||h_0||_{\infty} + 2C_{T_0} \sqrt{\mathcal{H}(F_0)-\mathcal{H}(\mu_E)+|M_0|+|E_0|}.
\end{eqnarray*}
For any $t>0$, we can find $n$ such that $nT_0 \leq t\leq \{n+1\}T_0$ and form $L^{\infty}$ estimate on $[0,T_0]$, we conclude (\ref{stability}) by
\begin{equation*}
||h(t)||_{\infty} \leq C_{T_0}||h(nT_0)||_{\infty} \leq C \left\{ \ ||h_0||_{\infty} + \sqrt{\mathcal{H}(F_0)-\mathcal{H}(\mu_E)+|M_0|+|E_0|} \ \right\}.
\end{equation*}
\\
\textbf{Acknowledgements. } The author thanks Raffaele Esposito and Rossana Marra for the reference of the unpublished \cite{Asano}. This research is supported in part by FRG07-57227.

\end{document}